\documentclass[11pt]{amsart}
\usepackage{bbm}
\usepackage{amsmath}  
\usepackage{amsthm}     
\usepackage{enumerate} 
\usepackage{physics}
\usepackage{mathtools}
\usepackage[margin=1.0in]{geometry} 
\DeclareMathOperator*{\argmin}{argmin}

\newcommand{\propref}{Proposition~\ref}
\newcommand{\thmref}{Theorem~\ref}
\newcommand{\lemref}{Lemma~\ref}

\newcommand{\reref}{Remark~\ref}

\newtheorem{remark}{Remark}
\newtheorem{prop}{Proposition}
\newtheorem{thm}{Theorem}
\newtheorem{lemma}{Lemma}
\newtheorem{definition}{Definition}
\numberwithin{remark}{section}
\numberwithin{prop}{section}
\numberwithin{thm}{section}
\numberwithin{lemma}{section}
\numberwithin{definition}{section}
\numberwithin{equation}{section}
\title[]{Embedding of Walsh Brownian Motion}
\author[]{Erhan Bayraktar} \thanks{This research was supported in part by the National Science Foundation under grants DMS-1613170.}  
\address{Department of Mathematics, University of Michigan}
\email{erhan@umich.edu}
\author[]{Xin Zhang} 
\address{Department of Mathematics, University of Michigan}
\email{zxmars@umich.edu}
\date{\today}
\keywords{Skorokhod embedding problem, Walsh Brownian motion, Excursion theory, Vallois' embedding}
\begin{document}
\maketitle

\begin{abstract}
Let $(Z,\kappa)$ be a Walsh Brownian motion with spinning measure $\kappa$. Suppose $\mu$ is a probability measure on $\mathbb{R}^n$. We characterize all the $\kappa$ such that $\mu$ is a stopping distribution of $(Z,\kappa)$. If we further restrict the solution to be integrable, we show that there would be only one choice of $\kappa$. We also generalize Vallois' embedding, and prove that it minimizes the expectation $\mathbb{E}[\Psi(L^Z_{\tau})]$ among all the admissible solutions $\tau$, where $\Psi$ is a strictly convex function and $(L_t^Z)_{t \geq 0}$ is the local time of the Walsh Brownian motion at the origin. 
\end{abstract}

\section{Introduction}

The Skorokhod embedding problem was formulated and solved by Skorokhod in 1961 \cite{MR0185620}. For a centered target distribution $\mu$ with finite second moment, one looks for a stopping time $\tau$ such that $B_{\tau} \sim \mu$, where $(B_t)_{t \geq 0}$ is the standard Brownian motion. In over fifty years, various solutions have been proposed, and some of them have been shown to have particular optimality properties. There is a large number of literature on this problem, but we will only mention a few of them that are related to our own work. For more detailed information, we will refer the reader to \cite{MR2068476} for a nice survey, to \cite{MR3639595} for the Skorokhod embedding's connection with optimal transport, and to \cite{MR2762363} for its application to Mathematical finance. 

Although the embedding problem for one-dimensional Brownian motion has been well studied, there are not many results in higher  dimensions (see e.g. \cite{MR580142}, \cite{RePEc:arx:papers:1711.02784}). As stated in \cite[Section 3.10]{MR2068476}, if we consider measures concentrated on the unit circle, only the uniform distribution can be embedded by means of an integrable stopping time. The main challenge is that the multi-dimensional Brownian motion does not visit points anymore. This motivates us to consider the embedding of Walsh Brownian motion, and it turns out that any $\mu \in \mathcal{P}(\mathbb{R}^n)$ can be embedded using this alternative.

Walsh Brownian motion is a singular diffusion in $\mathbb{R}^n$, which behaves like a one dimensional Brownian motion on each ray away from $\mathbf{0}$. Once it hits the origin, it is kicked away from $\mathbf{0}$ like a reflecting Brownian motion, and is assigned a random direction according to some given distribution $\kappa \in \mathcal{P}(\mathcal{S}^{n-1})$ (see e.g. \cite{MR1022917}, \cite{MR3188354}, \cite{AST_1978__52-53__37_0}). Here $\mathcal{S}^{n-1}$ denotes the unit sphere of $\mathbb{R}^n$. We will denote a Walsh Brownian motion by $(Z_t)_{t \geq 0}=(\Gamma_t, R_t)_{t \geq 0}$, where $\Gamma_t \in \mathcal{S}^{n-1}$ represents the direction and $R_t \in [0,+\infty)$  the radius. For a given probability measure $\mu \in \mathcal{P}(\mathbb{R}^n)$, we want to investigate if there is a stopping time $\tau < +\infty$ and a spinning measure $\kappa \in \mathcal{P}(\mathcal{S}^{n-1})$ such that  $Z_{\tau} \sim \mu$. We will assume that $\mu( \{ \mathbf{0}\}) <1$, without loss of generality, otherwise, we have a trivial solution $\tau=0$. 

Define a map $\Phi$ from punctured Euclidean space $\mathbb{R}^n \setminus \{\bf{0}\}$ to $\mathcal{S}^{n-1} \times \mathbb{R}_+$ as the following, $$\Phi: z \mapsto (\gamma, r),$$ where $z= \gamma r$. 
Denote $\tilde{\mu}=\mu|_{\mathbb{R}^n \setminus \{\bf{0}\}} \circ \Phi^{-1} $, the pushforward measure of $\mu|_{\mathbb{R}^n \setminus \{\bf{0}\}}$. We extend $\tilde{\mu}$ to a probability measure on $\mathcal{S}^{n-1} \times [0, +\infty)$ by distributing the mass $\mu(\bf{0})$ to $\mathcal{S}^{n-1} \times \{0\}$ in proportion to $\gamma \mapsto \tilde{\mu}(\{\gamma\} \times \mathbb{R}_+)$ . Take $k=1-\tilde{\mu}(\mathcal{S}^{n-1} \times \mathbb{R}_+)=\mu( \{\bf{0}\})$ and assign mass $\frac{k}{1-k}\tilde{\mu}(A \times \mathbb{R}_+)$ to $A \times \{0\}$ for any Borel subset $A \subset \mathcal{S}^{n-1}$. Denote the first marginal of $\tilde{\mu}$ by $\tilde{\mu}_1$, the disintegration of $\tilde{\mu}$ with respect to $\tilde{\mu}_1$ by $(\tilde{\mu}_{\gamma})_{\gamma \in \mathcal{S}^{n-1}}$, and the barycenter on each ray by $m_{\gamma}:=\int_{r \geq 0} r  \ \tilde{\mu}_{\gamma}(d r)$. 

Our first result, Theorem~\ref{thm:first-main}, shows that $\mu$ is a stopping distribution if and only if  $\kappa \gg \tilde{\mu}_1$. It is proved by an application of potential theoretic results in \cite{MR0346920}. In particular, the proof is done by checking assumption \eqref{eq-1} which in general is quite difficult to verify. We verify this assumption by characterizing the $\alpha$-excessive functions of Walsh Brownian motion (see \propref{prop-5}). 

Subsequently, we prove that for $\mu$ to be embedded by means of integrable stopping times, it must have a finite second moment. Furthermore, the choice of spinning measure $\kappa$ turns out to be unique, i.e.,
\begin{equation*}\label{eq1}
\kappa( d \gamma)=\frac{m_{\gamma}}{m} \tilde{\mu}_1(d \gamma), \tag{$*$}
\end{equation*}
where $m=\int_{\mathbb{R}^n} \abs{z} \ \mu( d z)$. See Proposition~\ref{prop1} and Theorem~\ref{thm1}. Here we employ stochastic calculus on Walsh Brownian motion (established in \cite{MR3795064}, \cite{KARATZAS20191921}).

Next we consider a Vallois type embedding: Let $\Psi$ be a strictly convex function, and choose a spinning measure $\kappa$ as shown in \eqref{eq1}. Denote by $\mathcal{T}$ the collection of stopping times $\tau$ such that $(Z_{t \wedge \tau})_{t \geq 0}$ is uniform integrable and $Z_{\tau}$ is of distribution $\mu$. Using excursion theory, we also solve the optimization problem 
\begin{equation*}\label{eq5}
\inf\limits_{\tau \in \mathcal{T}} \mathbb{E}[\Psi(L^{Z}_{\tau})], \tag{$\star$}
\end{equation*}
where $(L_t^Z)_{t \geq 0}$ is the local time at the origin.

The rest of the paper is organized as follows. In Section 2, we briefly review the definition of Walsh Brownian motion and its stochastic calculus, and characterize its $\alpha$-excessive functions. In Section 3, we show the existence of almost-surely finite solutions. In Section 4, we provide a sufficient and necessary condition for the existence of integrable solutions. Finally in Section 5, we solve the optimization problem \eqref{eq5}.

\section{Walsh Brownian Motion and its $\alpha$-excessive functions}

For any function $f$ defined on a subset of $\mathbb{R}^n$, we use $f_{\gamma}(r)$ to represent $f(\gamma,r)$. The right-, left-hand  derivative with respect to the variable $r$ are written as $\partial_+f_{\gamma}(r)$, $\partial_-f_{\gamma}(r)$, respectively. In addition,  denote the origin by $\mathbf{0}$.

\subsection{Walsh Brownian motion}

\begin{definition}\label{def1}
Let $(R_t)_{t \geq 0}$ be a reflecting Brownian motion and $\tau_0:=\inf \{t \geq 0: \ R_t =0 \}$. A process $(Z_t)_{t \geq 0}$ is an $n$-dimensional Walsh Brownian motion with spinning measure $\kappa \in \mathcal{P}(\mathcal{S}^{n-1})$ if
\begin{enumerate}[(i)]
\item $Z_t=(\Gamma_t, R_t),$ where $\Gamma_t$ is a $\mathcal{S}^{n-1}$-valued process with the convention that $\Gamma_t=(1, 0, \dots, 0)$ when $R_t=0$;
\item If $Z_0=\bf{0}$, then for $t>0$, the random variable $\Gamma_t$ has distribution $\kappa$ independent of $R_t$;
\item If $Z_t=(\gamma,r)$ with $r>0$, then $\Gamma_t =\gamma$ on the set $\{ t < \tau_0\}$, and on the set $\{ t > \tau_0 \}$, $\Gamma_t$ has distribution $\kappa$ independent of $R_t$.
\end{enumerate}
\end{definition}

In \cite{MR1022917}, Barlow et al. proved the existence of Walsh Brownian motion, which will be briefly reviewed below. Denote the state space of $(Z_t)_{t \geq 0}$ by $$E:=\{(\gamma,r): \gamma \in \text{supp}(\kappa), \ r \geq 0 \}.$$    For $f \in \mathcal{C}(E)$, define
\begin{equation*}
\begin{aligned}
&\bar{f} (r)=\int_{ \gamma \in \mathcal{S}^{n-1}} f(\gamma,r) \ \kappa (d \gamma ), \ \ \  \text{for } r \geq 0. \\
\end{aligned} 
\end{equation*}
Let $(P^+_t)_{t \geq 0}$ be the semigroup of a reflecting Brownian motion on $[0, +\infty)$ and $(P^0_t)_{t \geq 0}$ be the semigroup of a Brownian motion on $[0,+\infty)$ killed at $0$.  By reflection principle, for any $g \in \mathcal{C}_0(\mathbb{R}_+)$ we have, $$\mathbb{E}_r[\mathbbm{1}_{\{\tau_0  \leq t\}}g(R_t)]=P^+_tg(r)-P^0_tg(r).$$
So intuitively for any $f \in \mathcal{C}_0(E)$ the following equations hold, 
\begin{equation*}
\begin{aligned}
\mathbb{E}_{(\gamma,r)}[\mathbbm{1}_{\{\tau_0 \leq t\}}f(Z_t)] & =P^+_t \bar{f}(r)-P^0_t \bar{f}(r), \\
\mathbb{E}_{(\gamma,r)}[\mathbbm{1}_{\{\tau_0 > t\}}f(Z_t)] &
=P^0_t f_{\gamma}(r),  \\
\end{aligned}
\end{equation*}
by which we could formally write down $P_t: \mathcal{C}_0(E) \to \mathcal{C}_0(E), t \geq 0$: 
\begin{equation*}
\begin{aligned}
P_t f (\gamma, 0) &=P^+_t \bar{f} (0), & \gamma \in \mathcal{S}^{n-1},             \\
P_t f(\gamma, r) &=P^+_t \bar{f}(r) + P^0_t (f _{\gamma}-\bar{f})(r),    & r>0. \\
\end{aligned}
\end{equation*}

It can be easily verified that $(P_t)_{t \geq 0}$ is a Feller semigroup on $\mathcal{C}_0(E)$ (see \cite[Theorem 2.1]{MR1022917}). According to \cite[III.7]{MR1331599}, there exits a strong Markov, $E$-valued c\`{a}dl\`{a}g process $(Z_t)_{t \geq 0}$ with transition function $(P_t)_{t \geq 0}$. Write $Z_t=(\Gamma_t, R_t)$ in polar coordinates, and set $\Gamma_t=(1, 0, \dots, 0)$ when $R_t=0$. Then the radial process $(R_t)_{t \geq 0}$ is a reflecting Brownian motion, and $\Gamma_t$ is constant on the set $\{ t < \tau_0\}$. It also can be shown that $(Z_t)_{t \geq 0}$ is continuous and a Feller diffusion on $E$ (see \cite[Theorem 2.4, Corollary 2.5]{MR1022917}). 

It remains to prove that $\Gamma_t$ has the distribution $\kappa$ independent of $R_t$, which is a result from excursion theory of It\^{o} \cite{MR0402949}. 
We refer to \cite[Section 2]{MR3188354} for a detailed proof, and here we only introduce some basic results of the excursion theory that we will use later.

The excursion space for $(Z_t)_{t \geq 0}$ is given by $$\mathcal{U}_Z=\mathcal{S}^{n-1} \times \mathcal{U}_R,$$
where $\mathcal{U}_R=\{e \in \mathcal{C}([0, +\infty)):  e^{-1}(0,+\infty)=(0,\xi), \text{ for some } \xi >0 \}$ is the excursion space of reflecting Brownian motion. We can associate $(Z_t)_{t \geq 0}$ with a Poisson point process on $ \mathcal{U}_Z$. To see this, let $(L_t^Z)_{t \geq 0}$ denote the local time of $(Z_t)_{t \geq 0}$ at the origin. 
It characterizes the amount of time spent by $(Z_t)_{t \geq 0}$ at $\mathbf{0}$ and is just local time of $(R_t)_{t \geq 0}$ at $0$.  Take $(I_l)_{ l \geq 0}$ to  be the right continuous inverse of $(L_t^Z)_{t \geq 0}$. We ``label'' excursions using the local time at $\mathbf{0}$. 

\begin{definition}
The excursion point process is the process $(e_l)_{l \geq 0}$, defined with values in  $\mathcal{U}_Z$ by 
\begin{enumerate}[(i)]
\item if $I_l-I_{l-}>0$, then $e_l$ is the map $$t \mapsto \mathbbm{1}_{\{t \leq I_l-I_{l-}\}} Z_{I_{l-}+t};$$
\item if $I_l-I_{l-}=0$, then $e_l$ is the identically zero function.
\end{enumerate}
\end{definition}
It is shown by It\^{o} that this excursion point process is a Poisson point process, with intensity function given by a unique $\sigma$-finite measure $\eta$ on the excursion space $\mathcal{U}_Z$ (refer to \cite[Chapter XII]{MR1725357}, \cite[Chapter VI, Section 8]{MR1780932} for details). For any $\mathcal{U} \subset \mathcal{U}_Z$ and $l > 0$, we set $\mathcal{U}_l:=(0, l) \times \mathcal{U}$, and $$N^{\mathcal{U}_l}=\sum_{0<s<l} \mathbbm{1}_{\mathcal{U}}(e_s),$$
which is the number of excursions in $\mathcal{U}$ before local time $l$. It can be shown that $N^{\mathcal{U}_l}$ is a Poisson random variable with parameter $l \eta(\mathcal{U})$. According to our construction, the measure $\eta$ is the product $\kappa \times n$, where $n$ is the excursion measure for reflecting Brownian motion. We recall one important property of the measure $n$, which will be used later on in Section 5 (see e.g. \cite[Chapter XII, Exercise 2.10]{MR1725357}). 

\begin{lemma}\label{lem-2}
For every $x>0$, we have 
$$n(\{e \in \mathcal{U}_R: \sup\limits_{t \geq 0} e(t) \geq x\})=\frac{1}{x}.$$
\end{lemma}

%At the end of this subsection, we want to mention another characterization of Walsh Brownian motion. First recall the definition of tree-metric.
%\begin{definition}
%We define the tree-metric on the plane as follows:
%\begin{equation*}
%\begin{aligned}
%\rho(z_1,z_2):=(r_1+r_2)\mathbbm{1}_{\{\gamma_1 \not = \gamma_2 \}}+ |r_1-r_2| \mathbbm{1}_{\{\theta_1 =\theta_2\}},    \ \ z_1, z_2 \in E,
%\end{aligned}
%\end{equation*}
%where $(\gamma_1,r_1), (\gamma_2,r_2)$ are the expressions in polar coordinated of $z_1, z_2$, respectively.
%\end{definition}
%It can be seen that Walsh Brownian motion is continuous with respect to tree-metric, and the topology induced by tree-metric is finer than the one induced by Euclidean metric. 

\subsection{Stochastic calculus}
First we state the change of variable formula for Walsh Brownian motion (see \cite{MR3168934}, \cite{MR3795064},             \cite{KARATZAS20191921}), and then prove a simple lemma which will be used many times in the rest of the paper.

\begin{definition}
Let $\mathfrak{D}$ be the class of Borel-measurable functions $g: \mathbb{R}^n \to \mathbb{R}$, such that 
\begin{enumerate}[(i)]
\item For every $\gamma \in \mathcal{S}^{n-1}$, the function $r \mapsto g_{\gamma}(r)$ is differentiable  on $[0, +\infty)$, and the derivative $r \mapsto g_{\gamma}^{'}(r)$ is absolutely continuous on $[0,+\infty)$;
\item The function $\gamma \mapsto  g^{'}_{\gamma}(0)$ is  bounded.
\item There exist a real number $\xi>0$ and a Lebesgue-integrable function $\iota: (0, \xi] \to [0, +\infty)$ such that $|g^{''}_{\gamma}| \leq \iota(r)$ holds for all $\gamma \in \mathcal{S}^{n-1}$ and $r \in (0, \xi]$.
\end{enumerate}
\end{definition}

Now define a process $$B_t^Z=R_t-R_0-L_t^Z,$$ which is a Brownian motion according to \cite[Lemma 2]{MR1022917}. We have the following change of variable formula (see \cite[Theorem 2.12]{KARATZAS20191921}).

\begin{lemma}\label{lem10}
Let $(Z_t)_{t \geq 0}$ be a Walsh Brownian motion with spinning measure $\kappa$. Then for any $g \in \mathfrak{D}$, the process $g(Z_t)_{t \geq 0}$ is a continuous semimartingale and satisfies the identity
\begin{equation*}
\begin{aligned}
g(Z_t)  = & g(Z_0)+ \int_0^t \mathbbm{1}_{\{ R_s \not = 0\}} g^{'}_{\Gamma_s}(R_s) \ dB^Z_s \\
&  +\int_0^t \mathbbm{1}_{\{ R_s \not = 0\}} g^{''}_{\Gamma_s}(R_s) \ ds +V_g^Z(t),
\end{aligned}
\end{equation*}
where $V_g^Z(t)=\big(\int_{\gamma \in \mathcal{S}^{n-1}} \partial_+g_{\gamma}(0) \kappa (d \gamma) \big)L_t^Z.$
\end{lemma}

\begin{prop}\label{prop100}
Suppose $\rho: \mathcal{S}^{n-1} \to (0,+\infty)$ is bounded, and define the hitting time of $\rho$, $$\tau(\rho)=\inf \{ t \geq 0: Z_t=(\gamma, \rho(\gamma)), \ \forall \gamma \in \mathcal{S}^{n-1} \}.$$
Then we have 
$$\mathbb{P}^{\mathbf{0}}[\Gamma_{\tau(\rho)} \in d \gamma]=\frac{\frac{1}{\rho(\gamma)}}{\int_{\beta \in \mathcal{S}^{n-1}} \frac{1}{\rho(\beta)} \kappa(d \beta)} \kappa (d \gamma),$$
$$\mathbb{E}^{\mathbf{0}}[\tau(\rho)]=\frac{\int_{\gamma \in \mathcal{S}^{n-1}} \rho(\gamma) \kappa(d \gamma)            }{    \int_{\gamma \in \mathcal{S}^{n-1}} \frac{1}{\rho(\gamma)} \kappa(d \gamma)        } .$$
\end{prop}
\begin{proof}
For any disjoint Borel subsets $A, B \subset \mathcal{S}^{n-1}$, define a measurable function on $\mathcal{S}^{n-1} \times \mathbb{R}_+$ as follows, 
$$h_{A,B}(\gamma,r)=(\kappa(A)\mathbbm{1}_B(\gamma)-\kappa(B)\mathbbm{1}_A(\gamma))r.$$
Applying \lemref{lem10}, we see that $(h_{A,B}(Z_t))_{t \geq 0}$ is a martingale. By the optional sampling theorem, we obtain 
\begin{equation*}
\begin{aligned}
0&=\mathbb{E}^{\mathbf{0}}[h_{A,B}(Z_{\tau(\rho)})]   \\
&=\kappa(A) \int_{\gamma \in B} \rho(\gamma) \mathbb{P}^{\mathbf{0}}[\Gamma_{\tau(\rho)} \in d \gamma ]-\kappa(B) \int_{\gamma \in A} \rho(\gamma) \mathbb{P}^{\mathbf{0}}[\Gamma_{\tau(\rho)} \in d \gamma ].
\end{aligned}
\end{equation*}
Since choices of $A$ and $B$ are  arbitrary, it can be see that $\frac{\rho(\gamma) \mathbb{P}^{\mathbf{0}}[\Gamma_{\tau(\rho)} \in d \gamma]}{\kappa( d\gamma)}$ is a constant, from which we can deduce the first part of the lemma. 

Take $g(\gamma,r)=r^2$. Again by \lemref{lem10}, we know that $(g(Z_t)-t)_{t \geq 0}$ is a martingale. By employing the optional sampling theorem, we conclude 
\begin{equation*}
\begin{aligned}
\mathbb{E}^{\mathbf{0}}[\tau(\rho)] &=\mathbb{E}^{\mathbf{0}}[g(Z_{\tau(\rho)})]=\int_{\gamma \in \mathcal{S}^{n-1}} \rho^2(\gamma) \mathbb{P}^{\mathbf{0}}[\Gamma_{\tau(\rho)} \in d \gamma]  \\
&= \frac{\int_{\gamma \in \mathcal{S}^{n-1}} \rho(\gamma) \kappa(d \gamma)            }{    \int_{\gamma \in \mathcal{S}^{n-1}} \frac{1}{\rho(\gamma)} \kappa(d \gamma)        } .\\
\end{aligned}
\end{equation*}
\end{proof}

\subsection{$\alpha$-excessive functions}
In this subsection, we characterize bounded $\alpha$-excessive functions of Walsh Brownian motion. 

\begin{definition}
Let $\alpha \geq 0$. A non-negative nearly Borel measurable function $g$ is called $\alpha$-excessive relative to $(Z_t)_{t \geq 0}$ if 
\begin{enumerate}[(i)]
\item  $g \geq e^{-\alpha t} P_t g$ for every $t \geq 0$; 
\item $e^{-\alpha t}P_t g \to g$ pointwise as $t \to 0$.
\end{enumerate}
\end{definition}

The characterization of $\alpha$-excessive functions for Brownian motion is well known. The following proposition can be deduced easily from \cite[Chapter II, 30]{MR1912205}, after which we present the result for Walsh Brownian motion.

\begin{lemma}\label{lem4}
Let $g: \mathbb{R} \to [0,+\infty)$ be an $\alpha$-excessive function of Brownian motion. Then there exists a non-negative concave function $W$ such that $g(x)=e^{-\sqrt{2 \alpha} x}W(e^{2 \sqrt{2 \alpha} x}).$
\end{lemma}
  
 \begin{prop}\label{prop-5}
 Suppose $g$ is a bounded $\alpha$-excessive function of $(Z_t)_{t \geq 0}$, then there exists a family of functions $(W_{\gamma})_{\gamma \in \mathcal{S}^{n-1}}$ such that
\begin{enumerate}[(i)]
\item The function $W_{\gamma}:[1,+\infty) \to [0,+\infty)$ is concave and $W_{\gamma}(1)=g(\mathbf{0})$; 
\item For $\gamma \in \mathcal{S}^{n-1}$, we have $g_{\gamma}(r)=e^{-\sqrt{2 \alpha} r}W_{\gamma}(e^{2 \sqrt{2 \alpha} r})$; 
\item $ \int_{\gamma \in \mathcal{S}^{n-1}} \partial_+W_{\gamma}(1) \ \kappa(d \gamma) \leq \frac{g(\mathbf{0})}{2}. $
\end{enumerate}
 \end{prop}
\begin{proof}
By the strong Markov property for Walsh Brownian motion and the definition of $\alpha$-excessive function, we have,  for all $s \leq t$
$$\mathbb{E}_z[e^{-\alpha t}g(Z_t) | \mathcal{F}_s]=e^{-\alpha s}\mathbb{E}^{Z_s}[e^{-\alpha (t-s)} g(Z_{t-s})]\leq e^{-\alpha s}g(Z_s).$$ So that $(e^{-\alpha t}g(Z_t))_{t \geq 0}$ is a supermartingale, and thus for any stopping time $\tau$, $$ g(z) \geq \mathbb{E}^z[e^{-\alpha \tau} g(Z_{\tau})].$$
By restricting $g$ to a single ray in $E$, we obtain that $g_{\gamma}$ is $\alpha$-excessive for Brownian motion for each $\gamma \in \mathcal{S}^{n-1}$. Using \lemref{lem4}, we have claimed $(i) \& (ii)$ of the theorem.   %Thus there exist $\epsilon_0 >0$ and $K \in \mathbb{R}$ such that for all $\gamma \in \mathcal{S}^{n-1}, \epsilon < \epsilon_0$, $$\frac{g_{\gamma}(\epsilon)-g(\mathbf{0})}{\epsilon} \geq K.$$

Take $\tau(\epsilon):=\inf \{ t \geq 0: R_t=\epsilon \}$.
Employing \propref{prop100}, we obtain that $\mathbb{E}^{\mathbf{0}}[\tau(\epsilon)]=\epsilon^2$ and $\mathbb{P}^{\mathbf{0}}[\Gamma_{\tau(\epsilon)} \in d \gamma]=\kappa(d \gamma)$. %Note that $\tau(\epsilon)$ is decreasing as $\epsilon \to 0$, and thus $0=\lim\limits_{\epsilon \to 0}\mathbb{E}^{\mathbf{0}}[\tau(\epsilon)]=\mathbb{E}^{\mathbf{0}}[\lim\limits_{\epsilon \to 0} \tau(\epsilon)]$ by the monotone convergence theorem. So that $\lim\limits_{\epsilon \to 0} \tau(\epsilon)=0$ almost surely.  
By the optional sampling theorem, we get $g(\mathbf{0}) \geq \mathbb{E}^{\mathbf{0}}[e^{-\alpha \tau(\epsilon)}g(Z_{\tau(\epsilon)})].$ Subtracting both sides by $g(\mathbf{0})$ and dividing the inequality by $\epsilon$, we obtain 
\begin{equation}\label{eq-2}
\begin{aligned}
0 & \geq \mathbb{E}^{\mathbf{0}}\bigg[\frac{g(Z_{\tau(\epsilon)})-g(\mathbf{0})}{\epsilon}\bigg]-\mathbb{E}^{\mathbf{0}}\bigg[\frac{1- e^{-\alpha \tau(\epsilon)}}{\epsilon} g(Z_{\tau(\epsilon)})\bigg] \\
&=\int_{\gamma \in \mathcal{S}^{n-1}} \frac{g(\gamma,\epsilon)-g(\mathbf{0})}{\epsilon} \kappa(d \gamma)- \mathbb{E}^{\mathbf{0}}\bigg[\frac{1- e^{-\alpha \tau(\epsilon)}}{\epsilon} g(Z_{\tau(\epsilon)})\bigg] \\
\end{aligned}
\end{equation}
Since function $g$ is bounded, we can obtain the following inequality,
\begin{equation}\label{equation1}
\begin{aligned}
 \lim\limits_{\epsilon \to 0} \mathbb{E}^{\mathbf{0}}\bigg[& \frac{1- e^{-\alpha \tau(\epsilon)}}{\epsilon}  g(Z_{\tau(\epsilon)})\bigg]  \leq   \lVert g\rVert_{\infty} \lim\limits_{\epsilon \to 0} \mathbb{E}^{\mathbf{0}}\bigg[\frac{1- e^{-\alpha \tau(\epsilon)}}{\epsilon} \bigg]   \\
\leq & \lVert g\rVert_{\infty} \lim\limits_{\epsilon \to 0}\mathbb{E}^{\mathbf{0}}\bigg[ \frac{\alpha \tau(\epsilon)}{\epsilon}\bigg]=\lVert g\rVert_{\infty} \lim\limits_{\epsilon \to 0} \alpha \epsilon =0. 
\end{aligned}
\end{equation}
As to the first term on the right-hand side of \eqref{eq-2}, we rewrite $ \frac{g_{\gamma}(\epsilon)-g(\mathbf{0})}{\epsilon}=\int_0^{\epsilon} \frac{g_{\gamma}^{'}(r)}{\epsilon} \ dr $ as a sum of two integrals 

$$-\frac{\int_0^{\epsilon}  \sqrt{2 \alpha}g_{\gamma}(r) \ dr }{ \epsilon}+\frac{2(e^{\sqrt{2\alpha}\epsilon}-1)}{\epsilon} \frac{\int_1^{e^{\sqrt{2\alpha}\epsilon}}  W^{'}_{\gamma}(x^2)  \ dx }{ e^{\sqrt{2\alpha}\epsilon}-1}.$$
Denote the first term by $-u_{\gamma}(\epsilon)$, and the second term by $\frac{2(e^{\sqrt{2\alpha}\epsilon}-1)}{\epsilon}v_{\gamma}(\epsilon)$. In conjunction with \eqref{equation1}, \eqref{eq-2} becomes
\begin{equation}\label{eq-3}
\begin{aligned}
\int_{\gamma \in \mathcal{S}^{n-1}} & v_{\gamma}(\epsilon)  \kappa(d \gamma) \\
  & \leq \frac{\epsilon}{2(e^{\sqrt{2\alpha}\epsilon}-1)} \bigg( \int_{\gamma \in \mathcal{S}^{n-1}} u_{\gamma}(\epsilon) \kappa(d \gamma)+ \lVert g\rVert_{\infty} \lim\limits_{\epsilon \to 0}\mathbb{E}^{\mathbf{0}}\bigg[ \frac{\alpha \tau(\epsilon)}{\epsilon}\bigg] \bigg). \\
\end{aligned}
\end{equation}
Since $g$ is non-negative, the derivative $W^{'}_{\gamma}(r)$ is non-negative. In addition, function $v_{\gamma}(\epsilon)$ increases as $\epsilon$ decreases to $0$. Therefore, we can apply monotone convergence theorem to the left-hand side of \eqref{eq-3}. Since the boundedness of $g$ implies the boundedness of $u_{\gamma}(\epsilon)$, we can apply bounded convergence theorem to the first integral on the right-hand side of \eqref{eq-3}. Letting $\epsilon$ decrease to $0$ in \eqref{eq-3}, we obtain
$$\int_{\gamma \in \mathcal{S}^{n-1}} \partial_+W_{\gamma}(1) \kappa(d \gamma) \leq  \frac{1}{2}g(\mathbf{0}).$$
\end{proof}

\section{Almost surely finite Solutions}
Suppose $\mu$ is a Borel measure on Euclidean space $\mathbb{R}^n$. We want to characterize all the spinning measures $\kappa \in \mathcal{P}(\mathcal{S}^{n-1})$, such that $\mu$ is a stopping distribution of $(Z, \kappa)$. Here $(Z,\kappa)$ represents the Walsh Brownian motion with spinning measure $\kappa$, and we say $\mu$ is a stopping distribution if and only if there exists a stopping time $\tau< +\infty$ such that $Z_{\tau} \sim \mu$. 

Note that if $\tilde{\mu}_1=\sum_{i=1}^n \delta_{\gamma_i}$ is sum of finitely many atoms, we can choose $\kappa$ to be any measure which charges all directions $\gamma_i$, i.e., $$\kappa( \gamma_i) >0,  i=1,\dotso,n.$$
Since $\kappa( \gamma_i) >0$, the process $(Z_t)_{t \geq 0}$ is recurrent on $\gamma_i$-ray. We can enlarge $\mathcal{F}$ such that  $\mathcal{F}_0$ is rich enough to support an independent variable. Define the stopping time,  
\begin{equation*}
\begin{aligned}
\tau= \inf \{ t >1: \ Z_t=X\},
\end{aligned}
\end{equation*} 
where $X$ is $\mathcal{F}_0$-measurable and of distribution  $\mu$. It is clear that $Z_{\tau} \sim \mu$. 

However, if $\tilde{\mu}_1$ is continuous, it is impossible to choose $\kappa$ such that $(Z_t)_{t \geq 0}$ is recurrent on each ray. Consequently, the above construction no longer works. 
In \cite{MR0346920}, Rost answers a general question about the existence of embedding for an arbitrary Markov process. Suppose $(X_t)_{t \geq 0}$ is a transient Markov process, $U^X$ is its potential operator, and $\nu_0 U^X, \nu_1 U^X$ are $\sigma$-finite. Then, for the initial distribution $X_0 \sim \nu_0$, there exists a stopping time $\tau$ such that $Z_{\tau} \sim \nu_1$ if and only if $\nu_1 U^X \leq \nu_0 U^X$. If $(X_t)_{t \geq 0}$ is not transient, it would be killed at an independent time with exponential distribution (with parameter $\alpha$), which results in $(X^{\alpha}_t)_{t \geq 0}$. Rost proved that $\nu_1$ is a stopping distribution of $(X_t)_{t \geq 0}$ starting with $\nu_0$, if and only if $\lim\limits_{\alpha \to 0}( \nu_1 U^{X^{\alpha}} - \nu_0 U^{X^{\alpha}}) \leq 0$. The function $U^X(f)$ is excessive for any measurable $f$, so the result can be reformulated as the following (see \cite[Theorem 4]{MR0346920}). 

\begin{lemma}\label{lem40}
Suppose $(X_t)_{t \geq 0}$ is a Markov process with such state space that is locally compact and completely separable. A necessary and sufficient condition for $\mu$ to be a stopping distribution of $(X_t)_{t \geq 0}$ starting with $\delta_{\bf{0}}$ is 
\begin{equation}\label{eq-1}
\downarrow \lim_{\alpha \to 0} \sup\limits_{1 \geq g \in \mathfrak{S}^{\alpha}} \langle \mu-\delta_{\bf{0}}, g \rangle=0, 
\end{equation}
where $\mathfrak{S}^{\alpha}$ is the set of $\alpha$-excessive functions of $(X_t)_{t \geq 0}$.
\end{lemma}

 If $\tilde{\mu}_1(A)>0$ for some set $A \subset \mathcal{B}(\mathcal{S}^{n-1})$, we must have $\kappa(A)>0$ in order to make $\mu$ a stopping distribution, that is,  $\kappa \gg \tilde{\mu}_1$ is a necessary condition.  We will show that it is also  sufficient  by checking \eqref{eq-1} in \lemref{lem40}.

Define $\tilde{\mu}_{\gamma}^{\alpha}$ to be the pushforward measure of $\tilde{\mu}_{\gamma}$ under the mapping $r \mapsto e^{2 \sqrt{2 \alpha} r}$. We make two observations: (1) As $\alpha \to 0$, the measure $\tilde{\mu}_{\gamma}^{\alpha}$ will concentrate on a neighborhood of $1$; (2) The condition $g \leq 1$ is equivalent to $W_{\gamma}(x) \leq \sqrt{x}$ for any $ \gamma \in \mathcal{S}^{n-1}$.  

\begin{thm}\label{thm:first-main}
If $\tilde{\mu}_1$ is absolutely continuous with respect to $\kappa$, then the equation \eqref{eq-1} holds for Walsh Brownian motion $(Z_t)_{t \geq 0}$.
As a result, $\mu$ is a stopping distribution of $(Z,\kappa)$ if and only if $\kappa \gg \tilde{\mu}_1$.
\end{thm}
\begin{proof}
Recall that we have the characterization of $\alpha$-excessive functions by \propref{prop-5}. Taking $g \equiv 1$, we see that $\sup\limits_{1 \geq g \in \mathfrak{S}^{\alpha}} \langle \mu-\delta_{\bf{0}}, g \rangle \geq 0$ for any $\alpha >0$. It is sufficient to show for any $\epsilon>0$, there exists $\alpha_0>0$ such that $ \langle \mu-\delta_{\bf{0}}, g \rangle < \epsilon$ for any $\alpha < \alpha_0$ and $1 \geq g \in \mathfrak{S^{\alpha}}$.

Since $\kappa \gg \tilde{\mu}_1$, we can find $K>\frac{8}{\epsilon}$ such that for any $A \subset \mathcal{B}(\mathcal{S}^{n-1})$, $\kappa(A) < \frac{1}{K}$ implies $\tilde{\mu}_1(A)< \frac{\epsilon}{8}$. Take $\delta= \frac{\epsilon}{4K}$, and  $\alpha_0$ such that for any $\alpha < \alpha_0$   $$\int_{\gamma \in \mathcal{S}^{n-1}} \tilde{\mu}_{\gamma}^{\alpha}([1, 1+\delta]) \ \tilde{\mu}_1( d \gamma ) > 1- \frac{\epsilon}{4}.$$ 
If $C:=g(\mathbf{0})>1-\epsilon$, we automatically have $\langle \mu-\delta_{\mathbf{0}}, g \rangle < 1-C <\epsilon$. So without loss of generality, we assume $C \leq 1- \epsilon$.

Since $W_{\gamma}$ is concave, it is upper bounded on the interval $[1,1+\delta]$ by the linear function $$C+\partial_+W_{\gamma}(1)(x-1).$$  In order to have $$C+\partial_+W_{\gamma}(1)(x-1) \leq \sqrt{x}, \ x \in [1,1+\delta],$$
the derivative $\partial_+W_{\gamma}(1)$ cannot be greater than $\frac{\sqrt{1+\delta}-C}{\delta}$.
Denote by $H$ the collection of $\gamma$ such that $\partial_+W_{\gamma}(1) > \frac{\sqrt{1+\delta}-C}{\delta}$, i.e., $$H:=\left\{ \gamma \in \mathcal{S}^{n-1}:\partial_+W_{\gamma}(1)  > \frac{\sqrt{1+\delta}-C}{\delta} \right\}.$$ By part $(iii)$ of \propref{prop-5} and Markov's inequality, we have $$\kappa(H) \leq \frac{C \delta}{2\sqrt{1+\delta}-2C} \leq \frac{\delta}{2\epsilon}=\frac{1}{8K}<\frac{\epsilon}{64}, $$
and therefore $$\tilde{\mu}_1(H)< \frac{\epsilon}{8}.$$ Denote 
$$F:=\{ \gamma \in \mathcal{S}^{n-1}: \frac{d \tilde{\mu}_1}{d\kappa} \leq K \}, $$
 $$G:=\mathcal{S}^{n-1}\setminus (H \cup F) \subset \{ \gamma \in \mathcal{S}^{n-1}: \partial_+W_{\gamma}(1)  \leq \frac{\sqrt{1+\delta}-C}{\delta} \}.$$ Note that $\int_{\gamma \in \mathcal{S}^{n-1}}  \frac{d \tilde{\mu}_1}{d\kappa}  \ \kappa(d \gamma)=1$. Therefore by Markov's inequality, we have $\kappa(G)<\frac{1}{K}$ and  thus $\tilde{\mu}_1(G) < \frac{\epsilon}{8}.$
According to our choice of $\delta$, it can be seen that
\begin{equation}\label{eq-10}
\begin{aligned}
\mu(g)& =\int_{\gamma \in \mathcal{S}^{n-1}} \tilde{\mu}_1( d \gamma) \int_{0}^{+\infty} g_{\gamma}(r) \ \tilde{\mu}_{\gamma}(dr) \\
&=\int_{\gamma \in \mathcal{S}^{n-1}} \tilde{\mu}_1( d \gamma) \int_{1}^{+\infty} \frac{W_{\gamma}(x)}{\sqrt{x}} \ \tilde{\mu}_{\gamma}^{\alpha}(dx) \\
&\leq \int_{\gamma \in \mathcal{S}^{n-1}} \tilde{\mu}_1( d \gamma) \int_{1}^{1+\delta} \frac{W_{\gamma}(x)}{\sqrt{x}} \ \tilde{\mu}_{\gamma}^{\alpha}(dx)+\frac{\epsilon}{4}. \\
\end{aligned}
\end{equation}
We estimate the term $\frac{W_{\gamma}(x)}{\sqrt{x}}$ in \eqref{eq-10}.  For $\gamma \in H$, we have inequalities $$\int_1^{1+\delta} \frac{W_{\gamma}(x)}{\sqrt{x}} \tilde{\mu}^{\alpha}_{\gamma}(dx) \leq \tilde{\mu}^{\alpha}_{\gamma}([1,1+\delta]) \leq 1,$$
and for $\gamma \not \in H$, 
\begin{equation*}
\begin{aligned}
\int_1^{1+\delta} \frac{W_{\gamma}(x)}{\sqrt{x}} \ \tilde{\mu}^{\alpha}_{\gamma}(dx) & \leq 
\int_1^{1+\delta} \frac{C+\partial_+W_{\gamma}(1)(x-1)}{\sqrt{x}} \ \tilde{\mu}_{\gamma}^{\alpha}(dx)  \leq C+\delta \partial_+W_{\gamma}(1). \\
\end{aligned}
\end{equation*}
Therefore, we obtain the upper bound,
$$\mu(g) \leq C+ \int_{\gamma \in  \mathcal{S}^{n-1} \setminus H} \delta \partial_+W_{\gamma}(1)\  \tilde{\mu}_1( d \gamma) + \tilde{\mu}_1(H)+\frac{\epsilon}{4} .$$
For $\gamma \in F$, we have $$\tilde{\mu}_1(d \gamma) \leq K \kappa( d \gamma),$$ and for $\gamma \in G$, $$\partial_+W_{\gamma}(1) \leq \frac{(\sqrt{1+\delta}-C)}{\delta}.$$ 
In conjunction with part $(iii)$ of \propref{prop-5},  we get 
\begin{equation*}
\begin{aligned}
  \int_{\gamma \in  \mathcal{S}^{n-1}  \setminus H} & \delta \partial_+W_{\gamma}(1)\  \tilde{\mu}_1( d \gamma)  \\
 &\leq  \int_{\gamma \in F} \delta \partial_+W_{\gamma}(1) K \ \kappa(d \gamma) + \int_{\gamma \in G} \frac{\delta(\sqrt{1+\delta}-C)}{\delta} \ \tilde{\mu}_1(d \gamma) \\
& \leq  \frac{\delta K C}{2}+2 \tilde{\mu}_1(G). \\
 \end{aligned}
 \end{equation*}
Now we can conclude the result,
\begin{equation*}
\begin{aligned}
\mu(g) & \leq C+\frac{\delta K C}{2}+2 \tilde{\mu}_1(G) +\tilde{\mu}_1(H)+\frac{ \epsilon}{4} < \delta_{\mathbf{0}}(g)+\epsilon.\\
\end{aligned}
\end{equation*}
\end{proof}

\section{Integrable Solutions} 
In this section, we characterize the spinning measure $\kappa$ such that there exists an integrable solution $\tau$ to the embedding problem. It turns out that the choice of $\kappa$ is unique. Then we construct a solution as the limit of hitting times. 

\begin{prop}\label{prop1}
If there exists a stopping time $\tau$ such that $Z_{\tau} \sim \mu$ and $\mathbb{E}[\tau]< +\infty$, we must have 
\begin{equation*}\label{eq1}
\kappa( d \gamma)=\frac{m_{\gamma}}{m} \tilde{\mu}_1(d \gamma), \tag{$*$}
\end{equation*}
where $m=\int_{\mathbb{R}^n} \abs{z} \ \mu( d z)$ is a positive constant. Moreover, the second moment of $\mu$ must be finite, i.e.,
\begin{equation*}\label{eq2}
\int_{\mathbb{R}^n} \abs{z}^2 \ \mu(dz) < +\infty. \tag{$**$}
\end{equation*}
\end{prop}
\begin{proof}
Suppose $\tau$ is a stopping time such that $\mathbb{E}[\tau] < + \infty$ and $Z_{\tau} \sim \mu$. For any disjoint Borel subsets $A, B \subset \mathcal{S}^{n-1}$, define a measurable function on $\mathcal{S}^{n-1} \times \mathbb{R}_+$ as follows, 
$$h_{A,B}(\gamma,r)=(\kappa(A)\mathbbm{1}_B(\gamma)-\kappa(B)\mathbbm{1}_A(\gamma))r.$$
Applying \lemref{lem10}, $h_{A,B}(Z_t)$ is a martingale. Therefore 
\begin{equation*}
\begin{aligned}
0 & =\mathbb{E}[h_{A,B}(Z_0)]=\mathbb{E}[h_{A,B}(Z_{\tau})] \\
&= \kappa(A) \int_B \tilde{\mu}_1( d \gamma) \int_{\mathbb{R}_+} r \ \tilde{\mu}_{\gamma}( d r)-\kappa(B) \int_A \tilde{\mu}_1 (d \gamma) \int_{\mathbb{R}_+} r \ \tilde{\mu}_{\gamma}( d r) \\
&=\kappa(A) \int_B m_{\gamma} \ \tilde{\mu}_1(d \gamma)- \kappa(B) \int_A m_{\gamma} \ \tilde{\mu}_1(d \gamma). \\
\end{aligned}
\end{equation*}
Since the choice of disjoint pair $(A,B)$ is arbitrary, there exists a constant $m$ such that for any $\gamma \in \mathcal{S}^{n-1}$, $$m_{\gamma}\tilde{\mu}_1(d\gamma)=m \kappa(d \gamma).$$ 
Integrating both sides of the above equation over $\mathcal{S}^{n-1}$, we get 
\begin{equation*}
\begin{aligned}
m &=\int_{\mathcal{S}^{n-1}} m \ \kappa(d \gamma) =\int_{\mathcal{S}^{n-1}} m_{\gamma} \ \tilde{\mu}_1(d \gamma) = \int_{\mathbb{R}^n} \abs{z} \ \mu(dz). \\
\end{aligned}
\end{equation*}
Take another measurable function on $\mathcal{S}^{n-1} \times \mathbb{R}_+$, $g(\gamma,r)=r^2.$ Applying \lemref{lem10}, it can be seen that $g(Z_t)-t$ is a martingale. Since $\mathbb{E}[\tau] < +\infty$, we can employ Doob's optional sampling theorem and get 
$$\mathbb{E}[\tau]=\mathbb{E}[g(Z_{\tau})]=\int_{\mathbb{R}^n} \abs{z}^2 \ \mu(d z)< +\infty.$$
\end{proof}

\begin{remark}
The proof of \propref{prop1} implies that $m=\int_{\mathbb{R}^n} \abs{z} \ \mu(dz) < +\infty$. Since $m_{\gamma}$ is always positive, we see that $supp(\tilde{\mu}_1) \subset supp(\kappa)$. And for $\gamma \in supp(\kappa) \setminus supp(\tilde{\mu}_1)$, we must have $m_{\gamma}=+\infty$ which implies it is impossible to stop the process at $(\gamma, m_{\gamma})$.
\end{remark}

In the case of $n=1$,  $S^0$ consists of two directions $\{-, +\}$, and the process $(Z_t)_{t \geq 0}$ becomes a skew Brownian motion. Usually in the Skorokhod embedding problem, we say a target distribution $\mu$ is centered if  $\int_{-\infty}^{+\infty} x \ \mu(dx)  = 0$. Since the spinning measure of Brownian motion is $\kappa(+)=\kappa(-)=\frac{1}{2}$, it can be seen that  $\mu$ is centered if and only if $\tilde{\mu}_1(+) m_+= \tilde{\mu}_1(-)m_-$, which is equivalent to \eqref{eq1}. We generalize the concept of centered to Walsh Brownian motion, which is actually the property of both $\mu$ and $\kappa$. 
\begin{definition}
A pair $(\mu, \kappa)$ is said to be centered if they satisfy \eqref{eq1}. 
\end{definition}

Dubins \cite{MR0234520} presented a different approach to the Skorokhod embedding problem three years after \cite{MR0185620}. Starting with $\mu_0=\delta_0$, he constructed a sequence of finite sets $S_n$. The crucial idea is to transport every element of $S_n$ to two adjacent points in $S_{n+1}$ each time. The stopping time can be easily constructed inductively as hitting times, $$\tau_n=\tau_{n-1}+\tau_{S_n} \circ \theta_{\tau_{n-1}},$$
where $(\theta_t)_{t \geq 0}$ is the shift operator. Note that $B_{\tau_n}$ is finitely supported, and by a careful choice of $S_n$, the limit $\lim\limits_{n \to +\infty} B_{\tau_n}$ is of distribution $\mu$. 

Now by the same spirit of \cite{MR0234520}, we show that \eqref{eq1} \&\eqref{eq2} are also sufficient. The idea is to first stop the process at the barycenter $(m_{\gamma})_{\gamma \in \mathcal{S}^{n-1}}$ on each ray. Then for each $\gamma$, we view the process as a Brownian motion starting with $m_{\gamma}$, and we can employ the construction for Brownian motion.

For any $\nu \in \mathcal{P}([0, +\infty))$,  we introduce a sequence of increasing finite sets $(A^{\nu}_n)_{n \geq 0}$ by induction. 
For simplicity, denote $m^{\nu}_{[a,b)}=\frac{\int_{[a,b)} r \ \nu(dr)}{\nu([a,b))}$ (If $\nu([a,b))=0$, simply take $m^{\nu}_{[a,b)}=a$).  Initialize $A^{\nu}_{0}=\{0, +\infty\}$. Suppose $A^{\nu}_l=\{0=a_1 \leq a_2 \leq \dotso \leq a_{2^{l}+1}=+\infty\}$. Let $$A^{\nu}_{l+1} =A^{\nu}_l \bigcup\limits_{i=1}^{2^{l}} \{ m_{[a_i,a_{i+1})}^{\nu}\},$$
that is the union of $A^{\nu}_l$ and barycenters over each interval $[a_i, a_{i+1}), i =1, \dotso ,2^{l}$. 

Define the stopping sets as
\begin{equation*}
\begin{aligned}
S^{\gamma}_l &:=\{r : r \in A^{\tilde{\mu}_{\gamma}}_l\setminus A^{\tilde{\mu}_{\gamma}}_{l-1}\}, \ l \geq 1; \\
S_l &:=\{(\gamma, r): \gamma \in \mathcal{S}^{n-1}, \ r \in S^{\gamma}_l \},\ l \geq 1. \\
\end{aligned}
\end{equation*}
Take $$ \tau_1=\tau_{S_1} \ \ \ \text{and} \ \ \ \tau_{l}=\tau_{l-1}+\tau_{S_l} \circ \theta_{\tau_{l-1}} \ \text{for} \ l \geq 1 ,$$
where $\tau_S:=\{ t \geq 0: Z_t \in S \}$.  Since $(\tau_l)_{l \geq 1}$ is increasing, $\tau:=\lim\limits_{l \to +\infty} \tau_l$ always exists. In \thmref{thm1} we  show that $\tau< +\infty $ almost surely and $Z_{\tau} \sim \mu$. Let us start with two auxiliary lemmas.

\begin{lemma}\label{lemma1}
Suppose $b=m^{\nu}_{[a_1,a_2)}, b_1=m^{\nu}_{[a_1,b)}, b_2=m^{\nu}_{[b,a_2)}$. Then following equations hold, 
\begin{equation*}
\begin{aligned}
 \nu([a_1,a_2)) \frac{b_2-b}{b_2-b_1}& =\nu([a_1,b)), \\
\nu([a_1,a_2))\big(b^2+(b_2-b)(b-b_1) \big)& =\nu([a_1,b))b_1^2+\nu([b,a_2))b_2^2. \\
\end{aligned}
\end{equation*}
\end{lemma}
\begin{proof}
Note that $b=\frac{\nu{[a_1,b)}}{\nu{[a_1,a_2)}}b_1+\frac{\nu{[b,a_2)}}{\nu{[a_1,a_2)}}b_2$. Therefore we have 
\begin{equation*}
\begin{aligned}
& \frac{b_2-b}{b_2-b_1}  =\frac{b_2-\frac{\nu{[a_1,b)}}{\nu{[a_1,a_2)}}b_1-\frac{\nu{[b,a_2)}}{\nu{[a_1,a_2)}}b_2}{b_2-b_1}=\frac{\nu{[a_1,b)}}{\nu{[a_1,a_2)}},  
\end{aligned}
\end{equation*}
and 
\begin{equation*}
\begin{aligned}
 b^2+(b_2-b)&(b-b_1)    =b_2b+b_1b-b_1b_2 \\ 
&= \frac{\nu{[a_1,b)}}{\nu{[a_1,a_2)}}b_1b_2+\frac{\nu{[b,a_2)}}{\nu{[a_1,a_2)}}b_2^2+\frac{\nu{[a_1,b)}}{\nu{[a_1,a_2)}}b_1^2+\frac{\nu{[b,a_2)}}{\nu{[a_1,a_2)}}b_1b_2 -b_1b_2 \\
&= \frac{\nu{[a_1,b)}}{\nu{[a_1,a_2)}}b_1^2+\frac{\nu{[b,a_2)}}{\nu{[a_1,a_2)}}b_2^2. \\
\end{aligned}
\end{equation*}
\end{proof}

\begin{lemma}\label{lemma2}
If $\int_{[0,+\infty)} r \ \nu(dr) < +\infty$ and $A^{\nu}_l=\{0=a_1 \leq a_2 \leq \dotso \leq a_{2^{l}+1}=+\infty\}$ is defined as above, then $\nu_{l+1}:=\sum_{i=1}^{2^l} \nu([a_i,a_{i+1})) \delta_{m^{\nu}_{[a_i,a_{i+1})}}
$ weakly converges to $\nu$. 
\end{lemma}
\begin{proof}
It is easily seen that $\lim\limits_{l \to +\infty}A_l^{\nu}$ is dense in $supp(\nu)$.  Therefore for any continuous bounded function $f \in \mathcal{C}_b([0,+\infty))$, we have 
\begin{equation*}
\begin{aligned}
 \int_{[0,+\infty)} f(r) \ \nu(dr)=& \int_{supp(\nu)} f(r) \ \nu(dr) \\
=& \lim\limits_{l \to +\infty} \sum_{i=1}^{2^l} \nu([a_i,a_{i+1})) f(m^{\nu}_{[a_i,a_{i+1})})=\lim\limits_{l \to +\infty} \nu_{l+1}(f),
\end{aligned}
\end{equation*}
which concludes the result.
\end{proof}

\begin{thm}\label{thm1}
Suppose $\mu$ satisfies \eqref{eq1} and \eqref{eq2}. Then $\tau$ constructed above satisfies $\mathbb{E}[\tau] < +\infty$ and $Z_{\tau} \sim \mu$.
\end{thm}
 \begin{proof}

 We show it by calculating distribution of $(Z_{\tau_l})_{l \geq 1}$ and expectations $(\mathbb{E}[\tau_l])_{l \geq 1}$. Note that by definition of the stopping sets, $S_1$ is just $\{(\gamma, m_\gamma): \ \gamma \in  \mathcal{S}^{n-1}\}$. Then by \propref{prop100}, we have 
 \begin{equation*}
 \begin{aligned}
 \mathbb{P}[\Gamma_{\tau_1} \in d \gamma]=& \frac{\frac{1}{m_{\gamma}}  \kappa(d \gamma)}{\int_{\mathcal{S}^{n-1}} \frac{1}{m_{\alpha}} \ \kappa(d \alpha)}=\frac{\frac{1}{m_{\gamma}}  \kappa(d \gamma)}{\int_{\mathcal{S}^{n-1}} \frac{1}{m} \ \tilde{\mu}_1(d \alpha)} \\
 =& \frac{m}{m_{\gamma}} \kappa(d \gamma)=\tilde{\mu}_1(d \gamma).\\
 \end{aligned}
 \end{equation*}
Therefore, we have $Z_{\tau_1} \sim \tilde{\mu}_1 \times \delta_{m_{\gamma}}$. By assumption \eqref{eq2}, we have $$\int_{ \mathcal{S}^{n-1}} m_{\gamma}^2 \ \tilde{\mu}_1 (d \gamma) \leq \int_{\mathbb{R}^n} \abs{z}^2 \ \mu(dz)< + \infty. $$ 
Applying Doob's optional sampling theorem to the martingale $\abs{Z_t}^2-t$, we obtain $$\mathbb{E}[\tau_1]=\int_{ \mathcal{S}^{n-1}} m_{\gamma}^2 \ \tilde{\mu}_1 (d \gamma) .$$

It can be seen that if $\mathbf{0} \in S_n$ for any $n \geq 1$, it would be absorbed there (see \reref{re2}). Therefore we have $\Gamma_{\tau_1}=\Gamma_{\tau_n}, n\geq 1$. In order to calculate the distribution of $Z_{\tau_n}, n \geq 2$, we only need its conditional distribution on each ray.
Suppose $A^{\tilde{\mu}_{\gamma}}_l=\{0=a^{\gamma}_1<a^{\gamma}_2<\dotso<a^{\gamma}_{2^l+1}=+\infty\}, \gamma \in\mathcal{S}^{n-1}$. We prove by induction that 
\begin{equation}\label{eq3}
\mathbb{P}[Z_{\tau_{l+1}}=(\gamma, m^{\tilde{\mu}_{\gamma}}_{[a^{\gamma}_i, a^{\gamma}_{i+1})})|\Gamma_{\tau_{l+1}}=\gamma]=\tilde{\mu}_{\gamma}([a^{\gamma}_i,a^{\gamma}_{i+1})), \ i=1,\dotso,2^l, \ \gamma \in \mathcal{S}^{n-1},  \tag{I} 
\end{equation}
\begin{equation}\label{eq4}
 \mathbb{E}[\tau_{l+1}]=\int_{\mathcal{S}^{n-1}} \sum_{i=1}^{2^l} \tilde{\mu}_{\gamma}([a^{\gamma}_i,a^{\gamma}_{i+1})) (m^{\tilde{\mu}_{\gamma}}_{[a^{\gamma}_i,a^{\gamma}_{i+1})})^2  \ \tilde{\mu}_1(d\gamma). \tag{II}
\end{equation}

The first part of the proof has shown it is true for $l=0$. Assume the claim holds for $l \geq 1$. We calculate the conditional distribution of $Z_{\tau_{l+1}}$ and $\mathbb{E}[\tau_{l+1}]$. 

It is easily seen that $$A_{l+1}^{\tilde{\mu}_{\gamma}}=\{0=a_1^{\gamma}\leq b_1^{\gamma} \leq a_2^{\gamma}<b_2^{\gamma}\leq \dotso \leq a_{2^l+1}^{\gamma}=+\infty \} , \ \gamma \in \mathcal{S}^{n-1}, $$
where $b_i^{\gamma}=m_{[a_i^{\gamma},a^{\gamma}_{i+1})}^{\tilde{\mu}_{\gamma}}, \ i=1,\dotso, 2^l.$ Conditioning on $Z_{\tau_{l+1}}=(\gamma, b_i^{\gamma})$, $Z_{\tau_{l+2}}$ can only hit $(\gamma, m_{[a_i^{\gamma},b_i^{\gamma})}^{\tilde{\mu}_{\gamma}})$ or $(\gamma, m_{[b_i^{\gamma},a_{i+1}^{\gamma})}^{\tilde{\mu}_{\gamma}})$ by definition of $S_{l+2}^{\gamma}$. More precisely, due to strong Markov property of $(Z_t)_{t \geq 0}$, we obtain
\begin{equation*}
\begin{aligned}
\mathbb{P}& [Z_{\tau_{l+2}}= (\gamma,m^{\tilde{\mu}_{\gamma}}_{[a_i^{\gamma},b_i^{\gamma})})|Z_{\tau_{l+1}}
=(\gamma,b_i^{\gamma})] \\
&= \mathbb{P}[Z_{\tau_{l+1}+\tau_{S_{l+2}^{\gamma}} \circ \theta_{\tau_{l+1}}}=(\gamma,m^{\tilde{\mu}_{\gamma}}_{[a_i^{\gamma},b_i^{\gamma})})|Z_{\tau_{l+1}}=(\gamma,b_i^{\gamma})] =\mathbb{P}_{b_i^{\gamma}}[R_{\tau_{S^{\gamma}_{l+2}}}=m^{\tilde{\mu}_{\gamma}}_{[a_i^{\gamma},b_i^{\gamma})}], \\
\end{aligned}
\end{equation*}
where $(R_t)_{t \geq 0}$ behaves like Brownian motion before the hitting time. So it is the probability that a Brownian motion starting from $b_i^{\gamma}$ hits $m^{\tilde{\mu}_{\gamma}}_{[a_i^{\gamma},b_i^{\gamma})}$ before $m^{\tilde{\mu}_{\gamma}}_{[b_i^{\gamma},a_{i+1}^{\gamma})}$, which is well-known as $\frac{m^{\tilde{\mu}_{\gamma}}_{[b_i^{\gamma},a_{i+1}^{\gamma})}-b_i^{\gamma}}{m^{\tilde{\mu}_{\gamma}}_{[b_i^{\gamma},a_{i+1}^{\gamma})}-m^{\tilde{\mu}_{\gamma}}_{[a_i^{\gamma},b_i^{\gamma})}}$. Note that if $Z_t$ hits $m^{\tilde{\mu}_{\gamma}}_{[a_i^{\gamma},b_i^{\gamma})}$ at $\tau_{l+2}$, it must be at $b^{\gamma}_i$ at $\tau_{l+1}$. Therefore we conclude \eqref{eq3} by the following, 
\begin{equation*}
\begin{aligned}
\mathbb{P}[Z_{\tau_{l+2}}& = (\gamma,m^{\tilde{\mu}_{\gamma}}_{[a_i^{\gamma},b_i^{\gamma})}) |\Gamma_{\tau_{l+2}}=\gamma] \\
= & \mathbb{P}[Z_{\tau_{l+1}}=(\gamma,b_i^{\gamma})|\Gamma_{\tau_{l+1}}=\gamma]\ \mathbb{P}[Z_{\tau_{l+2}}=(\gamma,m^{\tilde{\mu}_{\gamma}}_{[a_i^{\gamma},b_i^{\gamma})})|Z_{\tau_{l+1}}=(\gamma,b_i^{\gamma})] \\
=&\tilde{\mu}_{\gamma}([a_i^{\gamma},a_{i+1}^{\gamma})) \      \frac{m^{\tilde{\mu}_{\gamma}}_{[b_i^{\gamma},a_{i+1}^{\gamma})}-b_i^{\gamma}}{m^{\tilde{\mu}_{\gamma}}_{[b_i^{\gamma},a_{i+1}^{\gamma})}-m^{\tilde{\mu}_{\gamma}}_{[a_i^{\gamma},b_i^{\gamma})}}=\tilde{\mu}_{\gamma}([a_i^{\gamma},b_i^{\gamma})), \\
\end{aligned}
\end{equation*}
where the last equation follows from \lemref{lemma1}.
Again by strong Markov property of $(Z_t)_{t \geq0}$, we obtain
\begin{equation*}
\begin{aligned}
 \mathbb{E}[\tau_{l+2}] =& \mathbb{E}[\tau_{l+1}]+\mathbb{E}[\mathbb{E}[\tau_{S_{l+2}} \circ \theta_{\tau_{l+1}}|Z_{\tau_{l+1}}]]  \\
=& \mathbb{E}[\tau_{l+1}]+ \int_{\mathcal{S}^{n-1}} \tilde{\mu}_1(d\gamma)
\big(\sum_{i=1}^{2^l} \tilde{\mu}_{\gamma}([a_i^{\gamma},a_{i+1}^{\gamma})) \mathbb{E}_{b_i^{\gamma}}[\tau_{S_{l+2}^{\gamma}}]                                             \big).
\end{aligned}
\end{equation*}
Note that $\mathbb{E}_{b_i^{\gamma}}[\tau_{S_{l+2}^{\gamma}}]$ is the expected exit time of a Brownian motion, starting from $b_i^{\gamma}$ and leaving the inteval $[m^{\tilde{\mu}_{\gamma}}_{[a_i^{\gamma},b_i^{\gamma})},m^{\tilde{\mu}_{\gamma}}_{[b_i^{\gamma},a_{i+1}^{\gamma})}]$. It is well-known as $(m^{\tilde{\mu}_{\gamma}}_{[b_i^{\gamma},a_{i+1}^{\gamma})}-b_i^{\gamma})( b_i^{\gamma}-m^{\tilde{\mu}_{\gamma}}_{[a_i^{\gamma},b_i^{\gamma})})$. Therefore we conclude \eqref{eq4} by using \lemref{lemma1},
\begin{equation*}
\begin{aligned}
\mathbb{E} [\tau_{l+2}]  
=& \int_{ \mathcal{S}^{n-1}} \tilde{\mu}_1(d\gamma) \bigg( \sum_{i=1}^{2^l} \tilde{\mu}_{\gamma}([a_i^{\gamma},a_{i+1}^{\gamma}))\big( (b_i^{\gamma})^2 + (m^{\tilde{\mu}_{\gamma}}_{[b_i^{\gamma},a_{i+1}^{\gamma})}-b_i^{\gamma})( b_i^{\gamma}-m^{\tilde{\mu}_{\gamma}}_{[a_i^{\gamma},b_i^{\gamma})})  \big)    \bigg) \\
=& \int_{ \mathcal{S}^{n-1}} \tilde{\mu}_1(d\gamma) \bigg(    \sum_{i=1}^{2^l} \big( \tilde{\mu}_{\gamma}([a_i^{\gamma},b_i^{\gamma}))(m^{\tilde{\mu}_{\gamma}}_{[a_i^{\gamma},b_i^{\gamma})})^2   +   \tilde{\mu}_{\gamma}([b_i^{\gamma},a_{i+1}^{\gamma}))(m^{\tilde{\mu}_{\gamma}}_{[b_i^{\gamma},a_{i+1}^{\gamma})})^2 \big) \bigg)
\end{aligned}
\end{equation*}

In the end, according to monotone convergence theorem and \lemref{lemma2}, we have $$\mathbb{E}[\tau]=\lim\limits_{l \to +\infty}\mathbb{E}[\tau_l]=\int_{\mathbb{R}^n} \abs{z}^2 \ \mu({dz}) < +\infty.$$ 
And since $(Z_t)_{t \geq 0}$ is continuous, the following equation holds for any bounded continuous $f \in \mathcal{C}_b(\mathbb{R}^n)$, 
$$\mathbb{E}[f(Z_{\tau})]=\lim\limits_{l \to +\infty}\mathbb{E}[f(Z_{\tau_l})]=\int_{\mathbb{R}^n} f(z) \ \mu(dz), $$
which implies $Z_{\tau} \sim \mu$.
 \end{proof}

\begin{remark}\label{re2}
If $0=b_i:=m^{\tilde{\mu}^{\gamma}}_{[a_i,a_{i+1})}$ for some $l \geq 1$, then $0=a_i=b_i=m^{\tilde{\mu}^{\gamma}}_{[b_i,a_{i+1})}$.  So that $0 \in S^{\gamma}_{l+1}$ and we have $\mathbb{E}[\tau_{l+1}| Z_{\tau_l}=(\gamma, b_i)]=\tau_l+\mathbb{E}_{0}[\tau_{S^{\gamma}_{l+1}}]=\tau_l.$ Therefore $Z_t$ may hit origin $\bf{0}$ but it would then be absorbed and would not visit other rays. 

\end{remark}

\section{A generalization of Vallois' Skorokhod embedding}
Chacon and Walsh \cite{MR0445598} gave a general construction of the Skorokhod embedding based on one dimensional potential theory. Later in \cite{MR2307397}, Cox and Hobson showed that the construction of both Az\'{e}ma-Yor \cite{MR544832} and Vallois \cite{MR1162722} can be interpreted in the framework of \cite{MR0445598}. For a strictly convex function $\Psi$, Vallois also proved that his solution minimizes $\mathbb{E}[\Psi(L_t^0)]$ among all the minimal solutions, where $(L_t^0)_{t \geq 0}$ is the local time of Brownian motion. Now we  generalize this result to Walsh Brownian motion by using the method established in \cite{MR0445598} and \cite{MR2307397}.

Suppose that the target distribution $\mu \in \mathcal{P}(\mathbb{R}^n)$ has a finite first moment, and that the spinning measure $\kappa$ is centered.
Let $\Psi$ be a strictly convex function such that $\Psi^{'}(+\infty) \leq K$ for some positive constant $K$. Let $\mathcal{T}$ be the collection of stopping times $\tau$ such that the stopped process $(Z_{t \wedge \tau})_{t \geq 0}$  is uniform integrable and $Z_{\tau}$ is of distribution $\mu$. Since the target distribution $\mu$ may not have finite second moment, stopping time $\tau \in \mathcal{T}$ does not have to be integrable (see \propref{prop1}).  We consider the optimization problem 
\begin{equation*}\label{eq5}
\inf\limits_{\tau \in \mathcal{T}} \mathbb{E}[\Psi(L^{Z}_{\tau})]. \tag{$\star$}
\end{equation*}

First we represent a sufficient condition for the uniform integrability mentioned above. Choose any $A \in \mathcal{B}(\mathcal{S}^{n-1})$ such that $1>\kappa(A) >0$, and recall $$h_{A,A^c}(\gamma,r)=(\kappa(A) \mathbbm{1}_{A^c}(\gamma)-\kappa(A^c)\mathbbm{1}_{A}(\gamma))r.$$ Define the hitting time for $x \in \mathbb{R}$,
\begin{equation}\label{ss}
H_x=\inf \{ t \geq 0: h_{A,A^c}(Z_t)=x \}. \tag{$\star$ $\star$}
\end{equation}
\begin{lemma}\label{lemmas}
If $ x \mathbb{P}[\tau > H_x] \to 0$ as $x \to \pm \infty$, then the stopped process $(Z_{t \wedge \tau})_{t \geq 0}$ is uniformly integrable.
\end{lemma}
\begin{proof}
The argument is part of \cite[Theorem 5]{MR2240692} and we repeat here for readers' convenience. Note that the uniform integrability of stopped processes $Z^{\tau}, R^{\tau}, h_{A,A^c}(Z^{\tau})$ are equivalent, and the process $(h_{A,A^c}(Z_t))_{t \geq 0}$ is a martingale, so it is sufficient to show that for any stopping time $\upsilon \leq \tau$, $$\mathbb{E}[h_{A,A^c}(Z_{\tau})| \mathcal{F}_{\upsilon}] = h_{A,A^c}(Z_{\upsilon}).$$

Suppose $x<0$, $F \in \mathcal{F}_{\upsilon}$, and set $F_x=F \cap \{\upsilon < H_x\}$. Since $h_{A,A^c}(Z_{t \wedge H_x})$ is a supermartingale, we have 
\begin{equation*}
\begin{aligned}
\mathbb{E}[h_{A,A^c}(Z_{\upsilon \wedge H_x})\mathbbm{1}_{F_x}] \geq \mathbb{E}[h_{A,A^c}(Z_{\tau \wedge H_x})\mathbbm{1}_{F_x}]. 
\end{aligned}
\end{equation*}
By replacing $B_t$ with $h_{A,A^c}(Z_t)$ in \cite[Lemma 9]{MR2240692}, we know $h_{A,A^c}(Z_{\upsilon})$ is integrable. As a result of the dominated convergence theorem, the left-hand side converges to $\mathbb{E}[h_{A,A^c}(Z_{\upsilon}) \mathbbm{1}_F]$ as $x \to -\infty$. It is noted that the term on the right is equal to 
$$\mathbb{E}[h_{A,A^c}(Z_{\tau})\mathbbm{1}_{F_x \cap \{\tau < H_x\}}]+x\mathbb{P}[F \cap \{\upsilon < H_x < \tau\}],$$
which converges to $\mathbb{E}[h_{A,A^c}(Z_{\tau})\mathbbm{1}_F]$ according to our assumption. We conclude $\mathbb{E}[h_{A,A^c}(Z_{\upsilon})\mathbbm{1}_F] \geq \mathbb{E}[h_{A,A^c}(Z_{\tau}) \mathbbm{1}_F]$, and thus $\mathbb{E}[h_{A,A^c}(Z_{\upsilon})| \mathcal{F}_{\upsilon}] \geq h_{A,A^c}(Z_{\tau}).$ By a same argument for $x >0$, we obtain the result.
\end{proof}

\subsection{Construction}

For each $\gamma \in \mathcal{S}^{n-1}$, take $s_{\gamma}(r)=\frac{mr}{m_{\gamma}}$ to be the scale function on rays. Define a family of functions on $[0, +\infty)$,
$$c_{\gamma}(r)=\frac{\int_0^{+\infty} \abs{s_{\gamma}(u)-r}  \ \tilde{\mu}_{\gamma}(du)  +r+m    }{2}.$$
Here $c_{\gamma}$ is our potential function on rays. It has the following properties (see e.g. \cite{MR0501374} for proofs). 

\begin{lemma}\label{lem20}
$c_{\gamma}$ is a positive convex function such that
\begin{enumerate}[(i)]
\item $c_{\gamma}(0)=m$ and $c_{\gamma}(r) \geq r$; 
\item $\partial_+c_{\gamma}(r)=\tilde{\mu}_{\gamma}([0,\frac{m_{\gamma}r}{m}])$, \ $\partial_-c_{\gamma}(r)=\tilde{\mu}_{\gamma}([0,\frac{m_{\gamma}r}{m}))$;
\item $\lim\limits_{r \to +\infty} c_{\gamma}(r)-r=0$.
\end{enumerate}
\end{lemma}

Take $\zeta_{\gamma}(s)$ to be the $r$-coordinate of the point on $c_{\gamma}$ where the tangent line passes through $(0,s)$. Since such a point may not be unique, we choose the one with maximum $r$-coordinate, i.e., 
\begin{equation*}%\label{sss}
\zeta_{\gamma}(s):=\sup \argmin_{r >0} \big\{ \frac{c_{\gamma}(r)-s  }{r} \big\}.  %\tag{$\star$ $\star$ $\star$}
\end{equation*}
We also take 
\begin{equation*}
\begin{aligned}
& \phi_{\gamma}(s) = \frac{c_{\gamma}(\zeta_{\gamma}(s))-s}{\zeta_{\gamma}(s)} , 
\ \ \Lambda(s)= \int_{\gamma \in \mathcal{S}^{n-1}} \phi_{\gamma}(s) \ \tilde{\mu}_1( d \gamma), \\
& H(s)=\int_0^s \frac{1}{\Lambda(u)} \ du, 
\ \ \ \ a_{\gamma}(l) =\frac{m_{\gamma}}{m}\zeta_{\gamma} (H^{-1}(l)). \\
\end{aligned}
\end{equation*}

We are now ready to define the stopping time, 
\begin{equation}\label{eq6}
\tau:=\inf \{ t \geq 0: R_t \geq a_{\Gamma_t}(L_t^Z)\},
\end{equation}
that is we stop the process if its excursion travels beyond the hypersurface $\gamma \mapsto a_{\gamma}(L_t^Z)$. 

Since $(a_{\gamma})_{\gamma \in \mathcal{S}^{n-1}}$ are non-increasing, stopping time $\tau$ is of barrier-type: taking $$B:=\underset{r \geq a_{\gamma}(l)}  \cup \{  [l,+\infty) \times \gamma r\} \subset [0,+\infty) \times \mathbb{R}^n,$$ we have $\tau=\inf\{t \geq 0: (L_t^Z,Z_t) \in B\}$. Before verifying that $\tau$ indeed embeds  measure $\mu$, we need a technical lemma. 

\begin{lemma}\label{lemma3}
$\phi_{\gamma}$ is absolutely continuous on closed subsets $[0,m)$ for each $\gamma \in \mathcal{S}^{n-1}$ and $$\phi_{\gamma}(s)=1-\int_0^s \frac{1}{\zeta_{\gamma}(u)} \ du .$$
\end{lemma}
\begin{proof}
The proof is from \cite[Lemma 2]{MR2307397}, and we record here for the sake of completeness. The function $ \phi_{\gamma}$ is the gradient of the tangent to $c_{\gamma}$ that passes through $(0,s)$. By the convexity of $c_{\gamma}$, we easily see that $\phi_{\gamma}$ is non-increasing on $[0, m)$. In addition, note that $c_{\gamma}$ is non-decreasing and $\zeta_{\gamma}$ is non-increasing. We estimate $\phi_{\gamma}(s-\delta)-\phi_{\gamma}(s)$ for small positive $\delta$, 
\begin{equation*}
\begin{aligned}
  \phi_{\gamma}(s-\delta)=&  \frac{c_{\gamma}(\zeta_{\gamma}(s-\delta))-(s-\delta)}{\zeta_{\gamma}(s-\delta)} \\
\leq &     \frac{c_{\gamma}(\zeta_{\gamma}(s))-s+\delta    }{\zeta_{\gamma}(s)}              =\phi_{\gamma}(s)+\frac{ \delta}{\zeta_{\gamma}(s)}.\\
\end{aligned}
\end{equation*}
Therefore $\phi_{\gamma}$ is $\frac{1}{\zeta_{\gamma}(s)}$-Lipschitz on closed intervals $[0,s] \subset [0,m)$ for any $s <m$. As a result, $\phi_{\gamma}$ is differentiable almost everywhere on $[0,m)$ and $\phi_{\gamma}(s)=\int_0^s {\phi}^{'}_{\gamma}(u) \ du +1$. 
Since $\zeta_{\gamma}$ is left-continuous, we can calculate the left derivative of $\phi_{\gamma}$, 

\begin{equation*}
\begin{aligned}
  \partial_-\phi_{\gamma}(s)  =& \partial_- \frac{c_{\gamma}(\zeta_{\gamma}(s))-s}{\zeta_{\gamma}(s)}\\
=& \frac{\partial_+c_{\gamma} (\zeta_{\gamma}(s)) \partial_- \zeta_{\gamma}(s)-1}{\zeta_{\gamma}(s)} -\frac{\partial_-\zeta_{\gamma}(s)(c_{\gamma}(\zeta_{\gamma}(s))-s) }{\zeta_{\gamma}^2(s)}\\
=& -\frac{1}{\zeta_{\gamma}(s)}+\frac{\partial_-\zeta_{\gamma}(s)}{\zeta_{\gamma}(s)}\bigg[\partial_+c_{\gamma}(\zeta_{\gamma}(s))-\frac{c_{\gamma}(\zeta_{\gamma}(s))-s}{\zeta_{\gamma}(s)}\bigg]. \\
\end{aligned}
\end{equation*}
If $\tilde{\mu}_{\gamma}$ has no atom at $\zeta_{\gamma}(s)$, $c_{\gamma}$ is then differentiable at $\zeta_{\gamma}(s)$ and $\partial_+c_{\gamma}(\zeta_{\gamma}(s))$ is just the gradient of the tangent $\frac{c_{\gamma}(\zeta_{\gamma}(s))-s}{\zeta_{\gamma}(s)}$. If $\tilde{\mu}_{\gamma}$ has an atom at $\zeta_{\gamma}(s)$, we know $\partial_-\zeta_{\gamma}(s)$ is zero. In both of these two cases, the second term of the above equation vanishes and we obtain the result. 
\end{proof}

\begin{thm}
The stopped process $(Z_{t \wedge \tau})_{t \geq 0}$ is uniformly integrable and $Z_{\tau}$ is of distribution $\mu$, where $\tau$ is defined in \eqref{eq6}.
\end{thm}
\begin{proof}
Our proof relies on the excursion theory (see e.g. \cite{MR1725357}, \cite{MR1780932}). It is noted that $L^Z_{\tau}$ is no less than $H(s)$ is equivalent to excursions at local time $l$ has maximum modulus less than $a_{\gamma}(l)$ for any $l \leq H(s)$, where $\gamma$ is the direction of excursions. Take a subset $\mathcal{V}$ of $\Pi:=[0,+\infty) \times \mathcal{S}^{n-1} \times \mathcal{U}_R$, $$\mathcal{V}:=\{(l, \gamma, e): \ l < H(s),\ \gamma \in \mathcal{S}^{n-1},  \ \sup\limits_{t \geq 0} e(t) \geq a_{\gamma}(l) \}.$$ 
According to \lemref{lem-2},   the random variable $N^{\mathcal{V}}=\sum\limits_{l >0} \mathbbm{1}_{\mathcal{V}}(l, e_l)$ is Poisson  with parameter $$\int_0^{H(s)} dl \ \int_{\gamma \in \mathcal{S}^{n-1}} \frac{\kappa(d \gamma)}{a_{\gamma}(l)}.$$
Since $L_{\tau}^Z \geq H(s)$ if and only if $N^{\mathcal{V}}=0$,  we obtain 
$$  \mathbb{P}[L^Z_{\tau} \geq H(s)]=exp \left\{ -\int_0^{H(s)} dl \ \int_{\gamma \in \mathcal{S}^{n-1}} \frac{\kappa(d \gamma)}{a_{\gamma}(l)} \right\}. $$
By \lemref{lemma3}, we have $$-\int_{\gamma \in \mathcal{S}^{n-1}} \frac{m \kappa( d \gamma)}{m_{\gamma}\zeta_{\gamma}(u)}
=-\int_{\gamma \in \mathcal{S}^{n-1}} \frac{\tilde{\mu}_1( d \gamma)}{ \zeta_{\gamma}(u)}
=\int_{\gamma \in \mathcal{S}^{n-1}} \phi_{\gamma}^{'}(u) \ \tilde{\mu}_1( d \gamma)=\Lambda^{'}(u).$$
In conjunction with $H^{'}(u)=\frac{1}{\Lambda(u)}$, we get
\begin{equation*}
\begin{aligned}
  \mathbb{P}[L^Z_{\tau} \geq H(s)]= & exp \left\{ -\int_0^{H(s)} dl \ \int_{\gamma \in \mathcal{S}^{n-1}} \frac{\kappa(d \gamma)}{a_{\gamma}(l)} \right\} \\
 =& exp \left\{-\int_0^s H^{'}(u) \ du \ \int_{\gamma \in \mathcal{S}^{n-1}} \frac{m \kappa(d \gamma)   }{m_{\gamma}\zeta_{\gamma}(u)}    \right\} \\
= & exp \left\{-\int_0^s \frac{\Lambda^{'}(u)}{\Lambda(u)} du                           \right\}=\Lambda(s). \\
\end{aligned}
\end{equation*}

Recall the definition of $\tau$: we will stop in the region $d \gamma \times [r,+\infty)$ at local time $l$ if and only if $(Z_t)_{t \geq 0}$ does dot hit stopping region $B$ until an excursion travels beyond $a_{\gamma}(l) \geq r$ at local time $l$. Take a subset of $\Pi$, $$\mathcal{V}^{'}:=\{(u,\gamma, e_u): \ u \in (l, l+dl], \ \sup\limits_{t \geq 0} e(t) \geq a_{\gamma}(l) \} .$$
Hence by \lemref{lem-2},  we obtain $$\mathbb{P}[N^{\mathcal{V^{'}}} \geq 1]=\frac{\kappa(d \gamma)}{a_{\gamma}(l)} dl.$$
Since $Z_{\tau} \in d\gamma \times [r,+\infty), L_{\tau}^Z \in dl$ if and only if $L_{\tau}^Z \geq l, N^{\mathcal{V}^{'}} \geq 1, a_{\gamma}(l) \geq r$, 
 we conclude $$\mathbb{P}[Z_{\tau} \in d\gamma \times [r,+\infty), L_{\tau}^Z \in dl ]=\mathbbm{1}_{\{ a_{\gamma}(l) \geq r\}}  \mathbb{P}[L_{\tau}^Z \geq l] \frac{\kappa(d \gamma)}{a_{\gamma}(l)} dl,   $$
and 
\begin{equation*}
\begin{aligned}
  \mathbb{P}[Z_{\tau} \in d\gamma \times [r, +\infty)] &= \int_{\{l: a_{\gamma}(l) \geq r \} } \mathbb{P}[L_{\tau}^Z \geq l] \frac{\kappa(d \gamma)}{a_{\gamma}(l)} \ dl \\
 &= \int_{\{u:  \zeta_{\gamma}(u) \geq \frac{mr}{m_{\gamma}}  \} } H^{'}(u) \Lambda(u) \frac{m \kappa(d \gamma)}{m_{\gamma}\zeta_{\gamma}(u)} \ du \\
 &= - \tilde{\mu}_1(d \gamma) \int_{\{u:  \zeta_{\gamma}(u) \geq \frac{mr}{m_{\gamma}}  \} } \phi^{'}_{\gamma}(u) \ du. \\
\end{aligned}
\end{equation*}

Since 
$\phi_{\gamma}(s)=\partial_-c_{\gamma}(\zeta_{\gamma}(s))=\tilde{\mu}_{\gamma}([0,\frac{m_{\gamma}\zeta_{\gamma}(s)}{m})),$
we obtain
$$\int_{\{u:  \zeta_{\gamma}(u) \geq \frac{mr}{m_{\gamma}}  \} } \phi^{'}_{\gamma}(u) \ du=\tilde{\mu}_{\gamma}([0,r))-1,$$
and therefore
$$\mathbb{P}[Z_{\tau} \in d \gamma \times [r, +\infty)]=\tilde{\mu}_1(d \gamma) \times \tilde{\mu}_{\gamma}([r, +\infty)).$$

To finish the argument, we show that $(Z_{t \wedge \tau})_{t \geq 0}$ is uniform integrable by verifying \lemref{lemmas}. Recall our notation from \eqref{ss}, and consider the case $x>0$. Due to the construction of $\tau$, we have 
\begin{equation*}
\begin{aligned}
\mathbb{P}[\tau> H_x] & =\int_{\gamma \in A^c} \kappa(d \gamma) \int_{\{l:a_{\gamma}(l) \geq x \}} \frac{\mathbb{P}[L^Z_{\tau} \geq l]}{x} dl \\
& =\int_{\gamma \in A^c} \kappa(d \gamma)  \int_{\{u:\zeta_{\gamma}(u) \geq \frac{mx}{m_{\gamma}} \}} \frac{H^{'}(u) \Lambda(u)}{x} \ du \\
\end{aligned}
\end{equation*}
Therefore we have $$x\mathbb{P}[\tau> H_x]=\int_{\gamma \in A^c} \text{Leb}(\{u: \zeta_{\gamma}(u) \geq \frac{mx}{m_{\gamma}}\}) \ \kappa(d \gamma).$$ Function $\gamma \mapsto \text{Leb}(\{u: \zeta_{\gamma}(u) \geq \frac{mx}{m_{\gamma}}\})$ is bounded above by $m$, and decreases to $0$ as $x \to + \infty$, so by the dominated convergence theorem, we have $\lim_{x \to +\infty} x \mathbb{P}[\tau> H_x]=0.$
\end{proof}

\begin{remark}
Our construction above is a solution of Barrier type. In the simplest  case when $(Z_t)_{t \geq 0}$ is a skew Brownian motion, we can use an alternative construction that relies on  scaling and time changing. According to our choice of $\kappa$,  $(Z_t)_{t \geq 0}$ must  be of parameter $$\kappa=\frac{\int_{\mathbb{R}_+} r \ \mu(dr)}{\int_{\mathbb{R}_+} r \ \mu(dr) - \int_{\mathbb{R}_-} r \ \mu(dr)}. $$
Denote the scale function as,
\[
s_{\kappa}(x)=
\begin{dcases}
2(1-\kappa)x        &   \text{ if }   x \geq 0, \\
2\kappa x             & \text{ if } x <0. \\
\end{dcases}
\]
Take $r_{\kappa}$ to be the inverse of $s_{\kappa}$. Let 
\[
\sigma_{\kappa}^2(x)=
\begin{dcases}
4(1-\kappa)^2  & \text{ if } x \geq 0, \\
4\kappa^2 & \text{ if } x<0.  \\
\end{dcases}
\]
Define the time change  $T_{\kappa}$ via
$$t=\int_0^{T_{\kappa}(t)} \frac{1}{\sigma_{\kappa}^2(W(u))} \ du.$$ 
Then $Z_t=r_{\kappa}(B_{T_{\kappa}(t)})$ is a skew Brownian motion of parameter $\kappa$ (see e.g. \cite{MR606993}, \cite{AST_1978__52-53__37_0}). We also have  $B_t=s_{\kappa}(Z_{A_{\kappa}(t)})$, where $A_{\kappa}(t):=\inf \{s : T_{\kappa}(s) >t \}$. 

It is well-known (see e.g. \cite[Theorem 9]{MR2307397}, \cite{MR1162722}) that there exists a pair of monotone function $(a,b)$ such that 
\begin{enumerate}[(i)]
\item $a: [0, +\infty) \to (-\infty, 0]$ is increasing, 
\item  $b:[0,+\infty) \to [0,+\infty)$ is decreasing, 
\item $\tau:=\inf \{ t : \ B_t \not \in (a(L^B_t), b(L^B_t)) \}$ embeds $\mu \circ s_{\kappa}^{-1}$, 
\end{enumerate}
Note that 
\begin{equation*}
\begin{aligned}
 L^Z_t  =& \lim_{\epsilon \to 0} \frac{1}{2 \epsilon} \int_0^t \mathbbm{1}_{(-\epsilon,+\epsilon)} (Z_s) \ ds \\
 = & \lim_{\epsilon \to 0} \frac{1}{2 \epsilon} \int_0^{T_{\kappa}(t)} \mathbbm{1}_{(-2\kappa \epsilon, 2 (1-\kappa) \epsilon)} B_s \ ds =L^B_{T_{\kappa}(t)}\\
\end{aligned}
\end{equation*}
Hence we have $$Z_{A_{\kappa}(\tau)}=r_{\kappa}(B_{\tau}) \sim \mu\circ s_{\kappa}^{-1} \circ r_{\kappa}^{-1}= \mu,$$ and also
\begin{equation*}
\begin{aligned}
 A_{\kappa}(\tau)=& \inf \{ A_{\kappa}(t): B_t \not \in  (a(L^B_t), b(L^B_t))\} \\
= & \inf \{t : B_{T_{\kappa}(t)} \not \in (a(L^B_{T_{\kappa}(t)}, b(L^B_{T_{\kappa}(s)}) \} \\
=& \inf \{t: Z_t \not \in (r_{\kappa} \circ a (L^Z_t), r_{\kappa} \circ b(L^Z_t)) \}, \\
\end{aligned}
\end{equation*}
which is of Barrier type.
\end{remark}

\subsection{Verification of Optimality.}
Beiglb\"{o}ck et al. have developed a new approach to the optimal skorokhod embedding problem based on the ideas of Optimal transport in \cite{MR3639595} and \cite{2019arXiv190303887B}, where the duality result and the monotonicity principle are presented. Most of their arguments are abstract and can carry over to the embedding problem for continuous Feller processes. By a similar argument as \cite[Theorem 6.14]{MR3639595}, we know that the optimizer of problem \eqref{eq5} must be of Barrier type. Since Barrier type solutions are in general essentially unique (see \cite{MR0292170}), our stopping time $\tau$ should solve the optimization problem \eqref{eq5}. 

Applying the method of pathwise inequalities established in \cite{MR3865340} and \cite{MR2462552}, we verify the optimality of $\tau$ by constructing the dual optimizer $(G, M)$. We define 
\begin{equation*}
\begin{aligned}
\Delta(l) & :=\int_0^l dm \int_{\gamma \in \mathcal{S}^{n-1}} \frac{1}{a_{\gamma}(m)} \ \kappa(d \gamma), \\
A_{\gamma}(l) &:=\Psi^{'}(+\infty)-\int_l^{+\infty} \frac{dm}{a_{\gamma}(m)} e^{\Delta(m)} \int_m^{+\infty} e^{-\Delta(n)} \Psi^{''}( dn).\\
% A(l)& := \int_{\gamma \in \mathcal{S}^{n-1}} A_{\gamma}(l) \ \kappa(d \gamma). \\
\end{aligned}
\end{equation*}
We now construct a function $G: \mathbb{R}^n \to \mathbb{R}$ and a local martingale $(M_t)_{t \geq 0}$ such that $M_t+ G(Z_t) \leq  \Psi(L^Z_t)$, and  equality is obtained when $Z_t=(\Gamma_t, a_{\Gamma_t}(L_t^Z))$. Define $G$ to be concave on each ray, 
\[
G(\gamma,r)=
\begin{dcases}
\Psi(0)- \int_0^{+\infty}  dm \ e^{\Delta(m)} \int_m^{+\infty} e^{-\Delta(n)} \Psi^{''}( dn)    & \text{ if } r=0,        \\
 \inf_{l >0} \left\{ r A_{\gamma}(l)+\Psi(0)- \int_0^{l}  dm \ e^{\Delta(m)} \int_m^{+\infty} e^{-\Delta(n)} \Psi^{''}( dn)  \right\}  & \text{ if } r >0. \\
\end{dcases}
\]
Denote by $b_{\gamma}$ the right-continuous inverses of $a_{\gamma}$. Since $a_{\gamma}(r)$ is non-increasing with respect to $r$, it is easily seen that the infimum above is obtained at $l=b_{\gamma}(r)$, and hence 
\begin{equation*}
\begin{aligned}
& G_{\gamma}^{'}(r+)   \leq A_{\gamma}(b_{\gamma}(r)) \leq G_{\gamma}^{'}(r-), \\
& G(\gamma,r) =rA_{\gamma}(b_{\gamma}(r))+\Psi(0)-\int_0^{b_{\gamma}(r)}  dm \ e^{\Delta(m)} \int_m^{+\infty} e^{-\Delta(n)} \Psi^{''}( dn). \\
\end{aligned}
\end{equation*}
Take $$M_t:=\int_0^{L_t^Z}dm \int_{ \gamma \in \mathcal{S}^{n-1}} A_{\gamma}(m) \ \kappa(d\gamma)-A_{\Gamma_t}(L_t^Z)R_t.$$ 
\begin{thm}
The random process $(M_t)_{t \geq 0}$ is a local martingale. We have  the pathwise inequality $$M_t+G(Z_t) \leq  \Psi(L_t^Z),$$ where equality is obtained for those paths such that $Z_t=(\Gamma_t, a_{\Gamma_t}(L_t^Z))$. 
\end{thm}
\begin{proof}
By \lemref{lem10}, we have 
\begin{equation*}
\begin{aligned}
-dM_t & =A^{'}_{\Gamma_t}(L_t^Z)R_t  \ d L_t^Z + \mathbbm{1}_{\{R_t \not = 0\}}A_{\Gamma_t}(L_t^Z) \ dB^Z_t \\
& + \int_{\gamma \in \mathcal{S}^{n-1}} A_{\gamma}(L_t^Z) \ \kappa( d \gamma) \ d L_t^Z 
 - \int_{\gamma \in \mathcal{S}^{n-1}} A_{\gamma}(L_t^Z) \ \kappa( d \gamma) \ d L_t^Z \\
& =A^{'}_{\Gamma_t}(L_t^Z)R_t  \ d L_t^Z + \mathbbm{1}_{\{R_t \not = 0\}}A_{\Gamma_t}(L_t^Z) \ dB^Z_t. \\
\end{aligned}
\end{equation*}
Since $(L_t^Z)_{t \geq 0}$ is flat off $\{R_t=0\}$, the first term vanishes. Therefore $(M_t)_{t \geq 0}$ is a local martingale. 

Note that 
\begin{equation*}
\begin{aligned}
\int_{\gamma \in \mathcal{S}^{n-1}} A_{\gamma}(l) \ \kappa( d \gamma) &=\Psi^{'}(+\infty)-\int_l^{+\infty} \Delta^{'}(m) e^{\Delta(m)} \ dm \int_m^{+\infty}  e^{-\Delta(n)} \Psi^{''}( dn) \\
&=\Psi^{'}(+\infty)-\int_l^{+\infty}  e^{-\Delta(n)} \Psi^{''}( dn) \int_l^{n} \Delta^{'}(m) e^{\Delta(m)} \ dm \\
&=\Psi^{'}(+\infty) -\int_l^{+\infty}  e^{-\Delta(n)}   \Psi^{''}( dn) (e^{\Delta(n)}-e^{\Delta(l)}) \\
&=\Psi^{'}(l)+e^{\Delta(l)}\int_l^{+\infty}  e^{-\Delta(n)}   \Psi^{''}( dn). \\
\end{aligned}
\end{equation*}
Therefore by the definition of $G$, we obtain 
\begin{equation*}
\begin{aligned}
G(Z_t) & \leq A_{\Gamma_t}(L_t)R_t+\Psi(0)- \int_{0}^{L^Z_t}  dm \ e^{\Delta(m)} \int_m^{+\infty} e^{-\Delta(n)} \Psi^{''}( dn)  \\
&=- M_t+\Psi(L_t^Z),
\end{aligned}
\end{equation*}
where the inequality is strict unless $Z_t=(\Gamma_t, a_{\Gamma_t}(L_t^Z))$.
\end{proof}

Suppose $\tau^{'}$ is a stopping time such that $Z_{\tau^{'}} \sim \mu$ and $(Z_{t \wedge \tau^{'}})_{t \geq 0}$ is uniformly integrable. Then by a similar argument as \cite[Lemma 2.1]{MR2462552}, we know that the stopped process 
$(M_{t \wedge \tau^{'}})_{t \geq 0}$ is uniformly integrable, and thus $\mathbb{E}[M_{t \wedge \tau^{'}}]=0$.  So that we have $$\int_{z \in \mathbb{R}^n} G(z) \ \mu(dz) \leq \mathbb{E}[\Psi(L^Z_{\tau^{'}})],$$
where equality is obtained when $\tau^{'}=\tau$. Therefore the stopping time $\tau$ solves the optimization problem \eqref{eq5}.

\footnotesize{
\bibliographystyle{siam}
\bibliography{ref.bib}}

\begin{thebibliography}{10}

\bibitem{MR544832}
{\sc J.~Az\'{e}ma and M.~Yor}, {\em Le probl\`eme de {S}korokhod:
  compl\'{e}ments \`a ``{U}ne solution simple au probl\`eme de {S}korokhod''},
  in S\'{e}minaire de {P}robabilit\'{e}s, {XIII} ({U}niv. {S}trasbourg,
  {S}trasbourg, 1977/78), vol.~721 of Lecture Notes in Math., Springer, Berlin,
  1979, pp.~625--633.

\bibitem{MR1022917}
{\sc M.~Barlow, J.~Pitman, and M.~Yor}, {\em On {W}alsh's {B}rownian motions},
  in S\'{e}minaire de {P}robabilit\'{e}s, {XXIII}, vol.~1372 of Lecture Notes
  in Math., Springer, Berlin, 1989, pp.~275--293.

\bibitem{MR3639595}
{\sc M.~Beiglb\"{o}ck, A.~M.~G. Cox, and M.~Huesmann}, {\em Optimal transport
  and {S}korokhod embedding}, Invent. Math., 208 (2017), pp.~327--400.

\bibitem{2019arXiv190303887B}
{\sc M.~{Beiglb{\"o}ck}, M.~{Nutz}, and F.~{Stebegg}}, {\em {Fine Properties of
  the Optimal Skorokhod Embedding Problem}}, arXiv e-prints,  (2019),
  p.~arXiv:1903.03887.

\bibitem{MR1912205}
{\sc A.~N. Borodin and P.~Salminen}, {\em Handbook of {B}rownian motion---facts
  and formulae}, Probability and its Applications, Birkh\"{a}user Verlag,
  Basel, second~ed., 2002.

\bibitem{MR0501374}
{\sc R.~V. Chacon}, {\em Potential processes}, Trans. Amer. Math. Soc., 226
  (1977), pp.~39--58.

\bibitem{MR0445598}
{\sc R.~V. Chacon and J.~B. Walsh}, {\em One-dimensional potential embedding},
  (1976), pp.~19--23. Lecture Notes in Math., Vol. 511.

\bibitem{MR3865340}
{\sc J.~Claisse, G.~Guo, and P.~Henry-Labord\`ere}, {\em Some results on
  {S}korokhod embedding and robust hedging with local time}, J. Optim. Theory
  Appl., 179 (2018), pp.~569--597.

\bibitem{MR2462552}
{\sc A.~M.~G. Cox, D.~Hobson, and J.~Ob\l\'{o}j}, {\em Pathwise inequalities
  for local time: applications to {S}korokhod embeddings and optimal stopping},
  Ann. Appl. Probab., 18 (2008), pp.~1870--1896.

\bibitem{MR2240692}
{\sc A.~M.~G. Cox and D.~G. Hobson}, {\em Skorokhod embeddings, minimality and
  non-centred target distributions}, Probab. Theory Related Fields, 135 (2006),
  pp.~395--414.

\bibitem{MR2307397}
\leavevmode\vrule height 2pt depth -1.6pt width 23pt, {\em A unifying class of
  {S}korokhod embeddings: connecting the {A}z\'{e}ma-{Y}or and {V}allois
  embeddings}, Bernoulli, 13 (2007), pp.~114--130.

\bibitem{MR0234520}
{\sc L.~E. Dubins}, {\em On a theorem of {S}korohod}, Ann. Math. Statist., 39
  (1968), pp.~2094--2097.

\bibitem{MR580142}
{\sc N.~Falkner}, {\em On {S}korohod embedding in {$n$}-dimensional {B}rownian
  motion by means of natural stopping times}, in Seminar on {P}robability,
  {XIV} ({P}aris, 1978/1979) ({F}rench), vol.~784 of Lecture Notes in Math.,
  Springer, Berlin-New York, 1980, pp.~357--391.

\bibitem{MR3188354}
{\sc P.~J. Fitzsimmons and K.~E. Kuter}, {\em Harmonic functions on {W}alsh's
  {B}rownian motion}, Stochastic Process. Appl., 124 (2014), pp.~2228--2248.

\bibitem{RePEc:arx:papers:1711.02784}
{\sc N.~Ghoussoub, Y.-H. Kim, and T.~Lim}, {\em {Optimal Brownian Stopping
  between radially symmetric marginals in general dimensions}}, Papers
  1711.02784, arXiv.org, Nov. 2017.

\bibitem{MR3168934}
{\sc H.~Hajri and W.~Touhami}, {\em It\^{o}'s formula for {W}alsh's {B}rownian
  motion and applications}, Statist. Probab. Lett., 87 (2014), pp.~48--53.

\bibitem{MR606993}
{\sc J.~M. Harrison and L.~A. Shepp}, {\em On skew {B}rownian motion}, Ann.
  Probab., 9 (1981), pp.~309--313.

\bibitem{MR2762363}
{\sc D.~Hobson}, {\em The {S}korokhod embedding problem and model-independent
  bounds for option prices}, in Paris-{P}rinceton {L}ectures on {M}athematical
  {F}inance 2010, vol.~2003 of Lecture Notes in Math., Springer, Berlin, 2011,
  pp.~267--318.

\bibitem{MR3795064}
{\sc T.~Ichiba, I.~Karatzas, V.~Prokaj, and M.~Yan}, {\em Stochastic integral
  equations for {W}alsh semimartingales}, Ann. Inst. Henri Poincar\'{e} Probab.
  Stat., 54 (2018), pp.~726--756.

\bibitem{MR0402949}
{\sc K.~It\^{o}}, {\em Poisson point processes attached to {M}arkov processes},
   (1972), pp.~225--239.

\bibitem{KARATZAS20191921}
{\sc I.~Karatzas and M.~Yan}, {\em Semimartingales on rays, walsh diffusions,
  and related problems of control and stopping}, Stochastic Processes and their
  Applications, 129 (2019), pp.~1921 -- 1963.

\bibitem{MR0292170}
{\sc R.~M. Loynes}, {\em Stopping times on {B}rownian motion: {S}ome properties
  of {R}oot's construction}, Z. Wahrscheinlichkeitstheorie und Verw. Gebiete,
  16 (1970), pp.~211--218.

\bibitem{MR2068476}
{\sc J.~Ob\l\'{o}j}, {\em The {S}korokhod embedding problem and its offspring},
  Probab. Surv., 1 (2004), pp.~321--390.

\bibitem{MR1725357}
{\sc D.~Revuz and M.~Yor}, {\em Continuous martingales and {B}rownian motion},
  vol.~293 of Grundlehren der Mathematischen Wissenschaften [Fundamental
  Principles of Mathematical Sciences], Springer-Verlag, Berlin, third~ed.,
  1999.

\bibitem{MR1331599}
{\sc L.~C.~G. Rogers and D.~Williams}, {\em Diffusions, {M}arkov processes, and
  martingales. {V}ol. 1}, Wiley Series in Probability and Mathematical
  Statistics: Probability and Mathematical Statistics, John Wiley \& Sons,
  Ltd., Chichester, second~ed., 1994.
\newblock Foundations.

\bibitem{MR1780932}
\leavevmode\vrule height 2pt depth -1.6pt width 23pt, {\em Diffusions, {M}arkov
  processes, and martingales. {V}ol. 2}, Cambridge Mathematical Library,
  Cambridge University Press, Cambridge, 2000.
\newblock It\^{o} calculus, Reprint of the second (1994) edition.

\bibitem{MR0346920}
{\sc H.~Rost}, {\em The stopping distributions of a {M}arkov {P}rocess},
  Invent. Math., 14 (1971), pp.~1--16.

\bibitem{MR0185620}
{\sc A.~V. Skorokhod}, {\em Studies in the theory of random processes},
  Translated from the Russian by Scripta Technica, Inc, Addison-Wesley
  Publishing Co., Inc., Reading, Mass., 1965.

\bibitem{MR1162722}
{\sc P.~Vallois}, {\em Quelques in\'{e}galit\'{e}s avec le temps local en zero
  du mouvement brownien}, Stochastic Process. Appl., 41 (1992), pp.~117--155.

\bibitem{AST_1978__52-53__37_0}
{\sc J.~B. Walsh}, {\em A diffusion with a discontinuous local time}, in Temps
  locaux, no.~52-53 in Ast\'erisque, Soci\'et\'e math\'ematique de France,
  1978, pp.~37--45.

\end{thebibliography}
\end{document}